\documentclass[12pt,a4paper,reqno]{amsart}
\usepackage{epsfig,amssymb,latexsym,amscd,amsmath,amsthm,stmaryrd,
  bussproofs}
\usepackage{verbatim}
\usepackage[all,2cell]{xy} \xyoption{v2} \UseAllTwocells
\usepackage{draftwatermark} \SetWatermarkLightness{0.95}
\SetWatermarkFontSize{5cm}

 \theoremstyle{plain}

\newtheorem{teo}{Theorem}[section] \newtheorem{theo}[teo]{Theorem}
 \newtheorem{lema}[teo]{Lemma}

\theoremstyle{remark}

\newtheorem{rema}[teo]{Remark}

\theoremstyle{definition}

\newtheorem{defi}[teo]{Definition} \newtheorem{obse}[teo]{Observation}

\newtheorem{example}[teo]{Example}

\usepackage[bbgreekl]{mathbbol} \usepackage[sans]{dsfont}
\usepackage{yfonts} \usepackage{stmaryrd}
\usepackage{euscript} \usepackage{mathrsfs}
\usepackage{txfonts} \usepackage[all]{xy}
\usepackage{graphics}

\renewcommand{\Perp}{\bot\!\!\!\bot}
\newcommand{\bbot}{\Perp}

 \renewcommand{\t}{t}
\newcommand{\koca}{\mathcal{{}^KOCA}} \newcommand{\oca}{\mathcal{OCA}}

\newcommand{\qoca}{\mathcal{{}^QOCA}} 

\renewcommand{\clubsuit}{\ensuremath{\Box}}
\newcommand{\clubsuitLittle}{\ensuremath{\Box}}

\newcommand{\cs}{{\scalebox{0.9}{$\mathsf {s}$\,}}}
\newcommand{\ck}{{\scalebox{0.9}{$\mathsf {k}$\,}}}
\newcommand{\ci}{{\scalebox{0.9}{$\mathsf {i}$\,}}}
\newcommand{\ce}{{\scalebox{0.9}{$\mathsf {e}$\,}}}

\newcommand{\cp}{{\scalebox{0.9}{$\mathsf {p}$\,}}}

\newcommand{\cS}{{\bf \mathsf {\scriptstyle{S}}\,}}
\newcommand{\cK}{{\bf \mathsf {\scriptstyle{K}}\,}}

\newcommand{\cB}{{\bf \mathsf {\scriptstyle{B}}\,}}

\newcommand{\cI}{{\bf \mathsf {\scriptstyle{I}}\,}}
\newcommand{\cE}{{\bf \mathsf {\scriptstyle{E}}\,}}

\newcommand{\ioca}{\ensuremath{\mathcal{{}^IOCA}}}
\newcommand{\foca}{\ensuremath{\mathcal{{}^FOCA}}}

\newcommand{\rapp}{\mathrm{app}} \newcommand{\rimp}{\mathrm{imp}}
\renewcommand{\int}{\operatorname{int}}
 
 \newcommand{\id}{\operatorname{id}}

\newcommand{\rpaxk}{(PK)}
\newcommand{\rpaxs}{(\hspace*{1pt}P\hspace*{.5pt}S\hspace*{.5pt})}
 \newcommand{\rpaxha}{(\hspace*{.5pt}PA)}

\begin{document}

\begin{abstract} In the context of the $\oca$ associated to an 
${\mathcal{AKS}}$ we introduce a closure operator and two associated
  maps that replace the closure and the maps defined in
  \cite{kn:ocar}. We were motivated by the search of a full adjunction
  to the original implication map. We show that all the constructions
  from $\oca$s to triposes developped in \cite{kn:ocar} can be also
  implemented in the new situation.

\end{abstract}
\title[Products in Abstract Krivine Structures and Ordered Combinatory Algebras.\\] {Realizability in $\oca${\tiny s} and
  ${\mathcal {AKS}}${\tiny s}} \author{Walter Ferrer Santos}
\address{Facultad de Ciencias\\Universidad de la Rep\'ublica\\ Igu\'a
  4225\\11400. Montevideo\\Uruguay\\} \email{wrferrer@cmat.edu.uy}
 
\author{Mauricio Guillermo} \address{Facultad de
  Ingenier\'ia\\Universidad de la Rep\'ublica\\ J. Herrera y Reissig
  565 \\ 11300. Montevideo \\ Uruguay\\} \email{mguille@fing.edu.uy}
\author{Octavio Malherbe} \address{Facultad de
  Ingenier\'ia\\Universidad de la Rep\'ublica\\ J. Herrera y Reissig
  565 \\ 11300. Montevideo \\ Uruguay\\} \email{malherbe@fing.edu.uy}
\thanks{The authors would like to thank Anii/FCE, EI/UdelaR and
  LIA/IFUM for their partial support} \today \maketitle
\section{Introduction}
\newcounter{bourbaki}
\renewcommand{\thebourbaki}{{\sf{\Roman{bourbaki}}}}
\begin{list}{\large{\sf{\Roman{bourbaki}.}}}{\usecounter{bourbaki}}
\bigskip
\item\label{item:intro} In this report we revisit some aspects of the
  important construction presented in the paper: \emph{Krivine's
    Classical Realizability from a Categorical Perspective} by Thomas
  Streicher --see \cite{kn:streicher}--.

\noindent
We put special emphasis in the algebrization of the structure of an
${\mathcal {AKS}}$ and the associated $\oca$ --see \cite{kn:report} or
\cite{kn:streicher} and the definitions and notations therein--.

\noindent
Streicher's paper points towards interpreting the classical
realizability of Krivine, as an instance of the categorical approach
started by Hyland--see \cite{kn:kr2001};\cite{kn:kr2003};
\cite{kn:kr2008}; \cite{kn:kr2009} for Krivine's approach and
\cite{kn:hyland} and \cite{kn:vanOosbook} for the categorical
approach. The present paper concentrates upon the basic algebraic set
up of these important categorical aspects of the theory.  The
categorical aspects will be considered in future work by the authors.
\bigskip
\item 
\vspace*{.2cm}
\noindent
In what follows we describe briefly the contents of each of the
sections of the current paper.  

\vspace*{.2cm}
\noindent
In \emph{Section
  \ref{section:ordered}} and in order to fix notations, we recall
general concepts on ordered structures and indexed ordered
structures. In particular, we write the definitions of preorder, meet
semilattice and Heyting preorder and its indexed versions i.e.,
contravariant functors from the category ${\bf Set}$ --of sets-- with
codomain the adequate subcategory of the category ${\bf Cat}$ --of all
categories --(e.g. the category of preorders, of meet semilatices, of
Heyting preorders). As a particular case of indexed Heyting preorders
we recall the crucial definition of tripos (see\cite{kn:hyland}).
\noindent
At the end, and as a tool for later usage, we briefly recall
how to program directly in an $\oca$ using the standard
codification in combinatory algebras (combinatory completeness).

\vspace*{.2cm}
\noindent
In \emph{Section \ref{section:oca}} we present the classes of ordered
combinatory algebras that we need in the paper, \emph{quasi
  implicative ordered combinatory}: $\qoca$s, \emph{implicative
  ordered combinatory algebras}: $\ioca$s and \emph{full adjunction
  implicative ordered combinatory algebras}: $\foca$s. We use the
completeness properties of $\mathcal A$ --the $\oca$ we are
considering-- to define a new product $\sharp: A \times A \rightarrow
A$ that is bounded bellow by the application. We also show how to
define realizability in high order arithmetics for $\foca$s (almost in
the same manner that we did in \cite{kn:ocar}, Section 6 for
$\ioca$s).

\vspace*{.2cm}
\noindent
In \emph{Section \ref{section:aksrl}} in preparation fot the definition of
Abstract Krivine Structure we introduce \emph{realizability lattices}
and some operations between the sets of terms and stacks of the
lattice. 

\noindent
We introduce the definition by steps. First we consider a pair of sets
equipped with a subset of its cartesian product and the induced
\emph{polarity} on the family of subsets of the original sets. We
consider two closure operators associated to the polarity and the
corresponding sets of closed subsets which are: the double
perpendicular introduced by Streicher in \cite{kn:streicher} and
another smaller one that we call $(-)\,\,{\widehat{}}$\,\,.  Then, if there is
an additional $\emph{push}$ map we say that we have a
\emph{realizability lattice} --similarly to the situation treated in
\cite[Section 2. Streicher's Abstract Krivine Structures]{kn:ocar}. We
define the basic operations of \emph{conductor} and \emph{push}--which
are the predecessors of the \emph{application} and the
\emph{implication} of the ${\mathcal{OCA}}$ and that are defined
between subsets of the set of stacks. All in all we define three
conductors and three push operations --including the ones used by
Streicher--, and the introduction of these operations will be
profitable for the understanding of the adjunction relations for
Streicher's operations, because the new operations are better behaved
with respect to adjunction than the original one.

\vspace*{.2cm}
\noindent
In \emph{Section \ref{section:opcomb}}, we add the remaining pieces of the
structure: the additional application operation in the set of terms,
the subset of terms called the set of quasi--proofs with its
distinguished elements, and the axioms relating the application with
the push through perpendicularity.

\noindent
We define the last operation, the application--like operation --called
the \emph{square} product and denoted as $\Box$-- and a distinguished
term called the ``adjunctor''. The adjunctor should be viewed as a
functor that will allow --partially, or in a ``lax'' sense-- the
recuperation of the \emph{full} adjunction property between the
conductor --application-- and the push --implication--, connecting
thus with Streicher's perspective in \cite{kn:streicher}.  Later, we
use the three pairs of application/implication--like operations to
construct different ${\mathcal {OCA}}$s starting from the same
${\mathcal {AKS}}$.  Only in Streicher's construction one needs the
adjunctor, and for this reason, the new $\oca$ --that we call
${\mathcal A}_{{\mathcal K},\bullet}$-- and that has stronger
adjunction properties, might be a good candidate to construct the
tripos we are searching for, in a simpler manner than the original
appearing in \cite{kn:streicher}.

%\vspace*{.2cm}
%\noindent 
%In \emph{Section \ref{section:akstohpo}} we begin by showing how the
%intermediate construction described in \cite{kn:ocar} between an
%$\mathcal{AKS}$ and a Heyting preorder (viewed as a category), can be
%performed in our new setting using a $\foca$ instead of a
%$\koca$. After reviewing our old construction we discuss strengths and
%benefits added by this new category whose objects are precisely
%described in Theorem \ref{theo:akstoocabullet}. The given
%$\operatorname{AKS}$ generates two equivalent (but not ispreordes that
%different preodifferent solutions of our fundamental problem it
%produces equivalent preorders. Next, we consider the opposite
%direction of this correspondence. We compare iteraded constructions
%that lead us to a Galois connection..

\vspace*{.2cm}
\noindent 
The content of \emph{Section \ref{section:akstohpo}} and \emph{Section
  \ref{section:constripos}} can be described using the diagram below
that summarizes the constructions appearing in \cite{kn:ocar}:

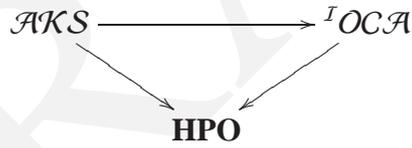
\[\xymatrix{\mathcal{AKS} \ar[rr] \ar[rd]&&\mathcal{{}^{\mathcal I}OCA}\ar[ld]\\
 &\bf {HPO}&}\] 
\noindent
 
In this diagram the concept of $\ioca$ and the construction of the right
diagonal map appears in \cite{kn:ocar}, the horizontal arrow
represents the construction of a $\ioca$ using the closure operator
associated to $\perp$ and the left diagonal arrow is Streicher's
construction as it appears in \cite{kn:streicher}.

Along these two sections, we reproduce the triangle of constructions
described above with the  modifications we mention below.

The $\ioca$s are substitued by the $\foca$s where no adjunctor is
needed and the new right diagonal arrow is formally the same
construction that the one in \cite{kn:ocar}, that can be performed in
the same manner thanks to presence of the filter. Moreover, in this paper the
horizontal arrow is the construction of a $\foca$ using the
closure operator $(-)\,\,\widehat{}$\,\,\,as described before.

Concerning the passage from ordered structures to triposes, that is
the main content of Section \ref{section:constripos} we need to
describe one more construction that is a crucial part of the theory
developped in \cite{kn:ocar}.  The upper arrow in the 
diagram below consists of the
building of an $\mathcal{AKS}$ from a $\ioca$ that was
performed in\,\cite[Section 5]{kn:ocar}.

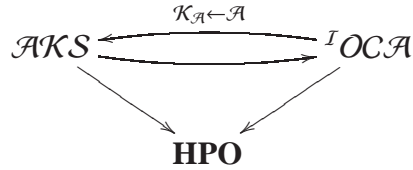
\[\xymatrix{\mathcal{AKS} \ar@/_.5pc/[rr]
  \ar[rd]&&\mathcal{\ioca}\ar[ld]\ar@/_.5pc/[ll]_{ {\mathcal
      K}_{\mathcal A}\leftarrow {\mathcal A}}\\ &{\bf HPO}&}\]

The content of \cite[Theorem 5.16]{kn:ocar}, is the following: if we
start from $\mathcal A \in \ioca$, take $\mathcal K_{\mathcal
  A}$ and perform on it Streicher's construction, the associated
tripos is equivalent to the one obtained directly from $\mathcal A$. 

In Theorem \ref{theo:equivtripos} we prove a version of the above
mentioned result starting now from a $\foca$ and performing similar
steps to the point where we prove an analogous result concerning the
equivalence of the associated triposes.
\bigskip

\item 
\vspace*{.2cm}
\noindent
In future work we intend to explore the use of the above methods in
order to explore the basic foundational pillars of the categorical
approach to realizability.
\bigskip

\section{Ordered structures, indexed ordered structures, triposes}
\label{section:ordered}
In this section --in order to fix notations-- we recall besides the
basic notion of tripos some general definitions related to ordered
structures and indexed ordered structures.

 \item \label{item:recall} We list some definitions and basic results. 
\begin{defi}
A preorder $(D,\leq)$ is a pair of a set $D$ with a reflexive and
transitive relation $\leq$. A morphism of preorders is a monotonic
map between the sets.
%A \emph{monotone map} between the preorders $(D,\leq)$ and $(E,\leq)$
%  is a function $f:D\to E$ such that $d\leq d'$ implies $f(d)\leq
%  f(d')$ for all $d,d'\in D$.
\begin{enumerate}
\item If $d\leq d'$, and $d'\leq d$, with $d,\,d'\in D$ we say that
  $d$ and $d'$ are \emph{isomorphic}, and write $d\cong d'$. 
\item If $(C,\leq)$ and $(D,\leq)$ are two preorders and
  $f,g:(C,\leq)\to(D,\leq)$ are morphisms of preorders, we say that
  $f\leq g :\Leftrightarrow \forall d\in D, f(d)\leq g(d)$ and that
  $f$ and $g$ are \emph{isomorphic} ($f\cong g$) if $f\leq g$ and
  $g\leq f$. A monotonic map $f:(C,\leq)\to(D,\leq)$ is an
  \emph{equivalence}, if there exists $g:(D,\leq)\to(C,\leq)$ monotonic
  such that $g\circ f\cong \id_C$, and $f\circ g\cong\id_D$. The map
  $g$ is called a \emph{weak inverse} of $f$. In this case we say that
  $(C,\leq)$ and $(D,\leq)$ are \emph{equivalent} (denoted as
  $(C,\leq)\simeq (D.\leq)$).    
\item Given monotonic maps
  $f:(C,\leq)\to(D,\leq)$, $g:(D,\leq)\to(C,\leq)$, we say that ``$f$
  is left adjoint to $g$'', or ``$g$ is right adjoint to $f$'', and
  write $f \dashv g$, if $\id_C\leq g\circ f$ and $f\circ g\leq
  \id_D$. 
\item We call ${\bf{Ord}}$ the category of preorders and, 
\emph{order preserving maps}, i.e.\ monotonic maps.
\item 
\begin{enumerate}
\item 
A \emph{principal filter} in $(D,\leq)$ is a subset of $D$ of the form
${\uparrow}d_0:=\{d \in D: d_0 \leq d\}$, for some $d_0\in D$.
\item
Dually, a \emph{principal ideal} in $D$ is a subset of the form
${\downarrow} d_0:=\{d\in D: d\leq d_0\}$, for $d_0\in D$. 
\end{enumerate}
\end{enumerate}
\end{defi}
\begin{obse}\label{obse:essentiallysur}\  
\begin{enumerate}
\item Given monotonic maps
  $f:(C,\leq)\to(D,\leq)$, $g:(D,\leq)\to(C,\leq)$, $f$ is left
  adjoint to $g$ if and only if 
  $\forall c\in C,\, d\in D,  f(c)\leq d \Leftrightarrow c\leq g(d)$.
\item We will use the following standard result in order theory: 

\noindent A monotonic map $f:(C,\leq)\to(D,\leq)$ is 
an equivalence if and only if: 
\begin{itemize}
\item $f$ is \emph{order reflecting}: $\forall
c,c'\in C, f(c)\leq f(c')\Rightarrow c\leq c'$
\item $f$ is \emph{essentially surjective}: $\forall d\in D\;\exists
  c\in C, 
  f(c)\cong d$. 
\end{itemize}
\item Concerning principal filters and principal ideals we have that:
$\inf({\uparrow}d_0)=d_0=\sup({\downarrow}d_0)$.  
\item The above standard notion of filter, is different from the one
  we use for $\oca$s later.
\end{enumerate}
\end{obse} 
\begin{defi} A \emph{meet semi-lattice} is a preorder $(A,\leq)$ equipped
  with a binary operation $\wedge$ and a distinguished element $\top$
  such that for all $a,b,c\in C$:
\begin{enumerate}
\item \label{item:meet1} $a\wedge b\leq a$; 
\item \label{item:meet2} $a\wedge b\leq b$; 
\item \label{item:meet3} $c\leq a\textrm{ and } c\leq b \Rightarrow c\leq
  a\wedge b$;
\item \label{item:top} $a\leq \top$. 
\end{enumerate}
We call ${\bf{SLat}}$ the category of meet semi-lattices,
and \emph{meet preserving monotonic maps}, i.e.\ monotonic maps
$f:(C,\leq)\to (D,\leq)$ such that $f(c)\wedge f(c')\cong f(c\wedge c')$ and $f(\top)\cong \top$.
\end{defi}

\begin{defi}\  
\begin{enumerate}
\item A \emph{Heyting preorder} is a meet semi-lattice $(A,\leq)$ with a
  binary operation $\to:A\times A\to A$ (called \emph{Heyting
    implication}) satisfying: 
\begin{equation}\label{eq:heyting-implication} \mbox{for all }
  a,b,c\in A\quad 
  a\wedge b\leq
c \,\,\,{\mathrm{if\,\,and\,\,only\,\,if}}\,\,\, a\leq b\to c
\end{equation} 
\item A \emph{morphism of Heyting preorders} is a monotonic map
  $f:(A,\leq)\to (B,\leq)$ such that $f(\top)\cong\top$, $f(a\wedge b)
  \cong f(a)\wedge f(b)$, $f(a\to b)\cong f(a)\to f(b)$ for all
  $a,b\in A$. 
\end{enumerate} 
We call ${\bf{HPO}}$, the category of
Heyting preorders and its morphisms.
\end{defi}
\begin{obse}
It follows from the definition above that in a Heyting preorder the
meet is monotonic in both arguments and the implication is 
antitonic in the first argument and monotonic in the second one.
\end{obse}
\item Next we generalize the definitions above, to the situation of
  indexed structures. We write down the definitions for the category
  of indexed preorders--for the other structures the definitions are
  similar (see \cite{kn:ocar}).
\begin{defi} An indexed category is a functor ${\bf F}:{\bf
    {Set}}^{\operatorname {op}}\rightarrow {\bf Cat}$, where ${\bf Cat}$
  is the category of all categories.
\end{defi} 
\begin{example} Some examples of indexed categories
    that we use in this work are the following --with special
    codomains inside $\bf Cat$:
\begin{enumerate}
\item An \emph{indexed preorder} is a functor
${\bf F}:{\bf {Set}}^{\operatorname {op}}\rightarrow {\bf{Ord}}$. 
\item An \emph{indexed meet-semi-lattice} is a functor
${\bf L}:{\bf {Set}}^{\operatorname {op}}\rightarrow {\bf{SLat}}$.
\item An \emph{indexed Heyting preorder} is a functor
${\bf P}:{\bf {Set}}^{\operatorname {op}}\rightarrow {\bf{HPO}}$.
\end{enumerate}

\end{example}
\begin{defi} Given indexed preorders 
${\bf C},{\bf D}:{\bf {Set}}^{\operatorname {op}}\rightarrow
  {\bf{Ord}}$ an \emph{indexed monotonic map} $\sigma:{\bf C} 
  {\rightarrow}{\bf D}$ is a family
$\sigma_I:{\bf C}(I)\to {\bf D}(I) \quad (I\in {\bf{Set}})$ of
  monotonic functions, such that for all functions $f:J{\to}I$ and 
predicates 
$\varphi\in{\bf C}(I)$: 
\begin{equation}\label{eq:pseudo} \sigma_J(f^*(\varphi))\cong
f^*(\sigma_I(\varphi))
\end{equation}
where the symbol $f^*$ denotes ${\bf C}(f)$ when on the left side of the equation and
  ${\bf D}(f)$ when on 
the right side.  
\end{defi}

\begin{defi} Let ${\bf C},{\bf D}:{\bf {Set}}^{\operatorname {op}}\rightarrow {\bf{Ord}}$ be indexed preorders.
\begin{enumerate}
\item For indexed monotonic maps $\sigma,\tau:{\bf C}\to{\bf
  D}$, we define
\[ \sigma\leq\tau\Leftrightarrow\forall I\in {\bf Set}\quad
\sigma_I\leq\tau_I \Leftrightarrow \forall I\in {\bf Set}
\,\,\text{and}\,\, \forall d \in {\bf C}(I)\,,\, \sigma_I(d) \leq
\tau_I(d).\] We say that $\sigma$ and $\tau$ are \emph{isomorphic},
and write $\sigma\cong \tau$, if $\sigma\leq \tau$ and $\tau\leq
\sigma$.
\item An indexed monotonic map $\sigma :{\bf C}\to{\bf D}$
  is called an \emph{equivalence}, if there exists a indexed monotonic
  map $\tau :{\bf D}\to {\bf C}$ such that $\tau \circ
  \sigma\cong \id_{\bf C}$, and $\sigma \circ
  \tau\cong\id_{\bf D}$. In this case, $\tau $ is called an
  \emph{(indexed) weak inverse} of $\sigma $.
\item We say that ${\bf C}$ and ${\bf D}$ are \emph{equivalent}, and write
${\bf C}\simeq{\bf D}$, if there exists an equivalence $\sigma
:{\bf C}\to{\bf D}$.
\end{enumerate}
\end{defi}
A proof of the next Lemma follows from the considerations appearing at
begining of Paragraph~\ref{item:recall}. 
\begin{lema}\label{lem:equiv-iord} An indexed monotonic map
$\sigma:{\bf C}\to{\bf D}$ is an equivalence, if and only if
  for every set $I$, the monotonic map $\sigma_I:{\bf
    C}(I)\to{\bf D}(I)$ is order reflecting and essentially
  surjective.
\end{lema}
\item Next we consider a special kind of indexed Heyting preorders,
  called \emph{triposes}, see \cite{kn:hyland}.
\begin{defi}\label{defi:triposes} 
A \emph{tripos} is a functor ${\bf P}:{\bf Set}^{\operatorname
  {op}}\to{\bf HPO} $ such 
that:
\begin{enumerate}
  \item For every function $f:J\to I$, the reindexing map
    $f^*:{\bf P}(I)\to{\bf P}(J)$ has a right adjoint
    $\forall_f:{\bf P}(J)\to{\bf P}(I)$.
\item If
\begin{equation}\label{eq:pullback}
\begin{array}[c]{ccc}
P&\stackrel{q}{\longrightarrow}&K\\
{\downarrow}\scriptstyle{p}&&{\downarrow}\scriptstyle{g}\\
J&\stackrel{f}{\longrightarrow}&I
\end{array}
\end{equation}

%% \begin{equation}\label{eq:pullback} \vcenter{\xymatrix{
%% P\pullbackcorner \ar[r]^q \ar[d]_p & K \ar[d]^g \\ J \ar[r]_f & I }}
%% \end{equation}
is a pullback square of sets and functions, then
$\forall_q(p^*(\varphi))\cong g^*(\forall_f(\varphi))$ for all
$\varphi\in{\bf P}(J)$ (this is the \emph{Beck-Chevalley condition}).
\item\label{def:tripos_chi} ${\bf P}$ has a \emph{generic
  predicate}, i.e.\ there exists a set ${\mathsf {Prop}}$, and a
  ${\mathsf {tr}}\in{\bf P}({\mathsf{Prop}})$ such that for every
  set $I$ and $\varphi\in{\bf P}(I)$ there exists a (not
  necessarily unique) function $\chi_\varphi:I\to{\mathsf{Prop}}$ with
  $\varphi\cong\chi_\varphi^*({\mathsf {tr}})$.

\end{enumerate}
\end{defi}

\begin{rema}\label{rema:equiv-iord-trip} Assume that ${\bf C}$ 
and ${\bf P}$ are equivalent indexed preorders, and that
${\bf P}$ is a tripos, then so is ${\bf C}$.
\end{rema}

\item \label{item:programing} We briefly recall how to program
  directly in an $\oca$ ${\mathcal A}$, using the standard
  codification in combinatory algebras--see for example
  \cite{kn:report} or \cite[Theorem 3.4]{kn:ocar} for details--.
\noindent
For any finite set of variables $\{x_1,\cdots,x_k,y\} $, there is a
function $\lambda^*y:A[x_1,\cdots,x_k,y] \rightarrow A[x_1,\cdots,
  x_k]$ satisfying the following property: If $t \in
A[x_1,\cdots,x_k,y]$ and $u \in A[x_1,\cdots,x_k]$, then $\lambda^*y
(t)\circ u \leq t\{y:=u\}$.
  
The function $\lambda^*y$ is defined recursively: i) If $y\neq x$,
then $\lambda^*y(x):= \ck x$; ii) $\lambda^*y(y):= \cs\ck\ck$; iii) if
$p, q$ are polynomials, then: $\lambda^*y(p q):= \cs(\lambda^*y(p))
(\lambda^*y(q))$.

Sometimes we write: $\lambda^*y(t)=\lambda^*y.t$.
\section{Ordered combinatory algebras}\label{section:oca} 
\noindent
\item \label{item:koca} We start by recalling the following basic
  structure, compare with \cite[Section 3.2]{kn:ocar}.

\begin{defi} \label{defi:ioca}\label{defi:foca}
%%An implicative ordered combinatory
  %% algebra --a $\ioca$--, consists 

Let $(A,\leq)$ be an $\operatorname{inf}$--complete
  partially ordered set equipped with: 
\begin{itemize}
\item[(OP)]\label{eq:OP:defi:ioca} binary operations
\[\rapp:A\times A\to A,\qquad(a,b)\;\mapsto\; a b\] called
\emph{application}, monotone in both arguments, and
\[\mathrm{imp}:A \times A\to A,\qquad (a,b)\;\mapsto\; a\to b\]
called \emph{implication}, antimonotone in the first argument and monotone
in the second;
\item[(CO)] distinguished elements $\cs,\ck \in A$
such that for all $a,b,c\in A$ the following holds:
\begin{itemize}
\item[\rpaxk] $\ck a b \leq a$
\item[\rpaxs] $\cs a b c\leq a c (b c)$
%\item[\rpaxc] $\comc\leq ((a\to b)\to a )\to a$
\item[\rpaxha] $\text{If}\quad a\leq b\to c, \quad \text{then}
  \quad \quad a b\leq c.$
 
\end{itemize}
\item[(FI)] a subset $\Phi\subseteq A$ (called \emph{filter}) which is
closed under application and such that $\cs, \ck \in \Phi$.
%\end{itemize}
\end{itemize} 
In above context we establish the following definitions:
\smallskip 
\begin{enumerate}
\item The structure $\mathcal A= (A, \operatorname{app},
  \operatorname{imp}, \cs, \ck, \Phi)$ as above is called a
  \emph{quasi implicative ordered combinatory algebra} --a $\qoca$.
\item If there exists a distinguished element  $\ce \in \Phi$ such that:       
\[\text{(PE)}\quad\quad\quad \text{If}\quad ab\leq
c,\quad\text{then}\quad \quad \ce a \leq b\to c \] 
 we say that $\mathcal A= (A, \operatorname{app}, \operatorname{imp},
 \cs, \ck, \ce, \Phi)$ is an \emph{implicative ordered combinatory
   algebra} --a $\ioca$.
\item If for all $a,b,c \in A$ \[\text{(PE)}'\quad\quad\quad
  \text{If}\quad ab\leq c,\quad\text{then}\quad \quad a \leq b\to
  c, \] we say that $\mathcal A= (A, \operatorname{app},
  \operatorname{imp}, \cs, \ck, \Phi)$ is a \emph{full adjunction
    implicative ordered combinatory algebra} --a $\foca$.
\end{enumerate}
\end{defi}  
\medskip

We refer the reader to \cite[Section 3]{kn:ocar} for motivation and
more details concerning the definition of $\ioca$, in particular we
mention that the element $\ce$ is called an \emph{adjunctor}.  When
there is no danger of confusion we denote $\mathcal A=
(A, \operatorname{app}, \operatorname{imp}, \cs, \ck, \ce, \Phi)$ as $A$.
\begin{obse} 
Notice that the element $\ci = \cs\ck\ck \in \Phi$, and that given a
$\foca$ called $A$, for all $a,b,c \in A$ if $ab \leq c$ implies that
$\ci a \leq a \leq b \rightarrow c$ and thus $A$ is also a $\ioca$
with adjunctor $\ci$.  
\end{obse} 
\begin{obse} 
    Notice that the application $\rapp(a,b)$ is denoted as
  $ab$ (c.f.: \ref{eq:OP:defi:ioca}(OP)). Eventually $ab$ will be denoted $a
    \circ b$, when there is danger 
  of confusion.

  When operating with the function $\rapp$ we associate to
  the left and when operating with the function $\rimp$ we associate
  to the right, so that $abc$ means $(ab)c$ and $a\to b\to c$ means
  $a\to (b\to c)$. 

\end{obse}
\item We use the completeness property to establish the following definition. 
\begin{defi} \label{defi:sharp}
Let $A$ be a $\qoca$, for $a,b \in A$ define 
\[a \,\sharp\, b=\inf\{c: a \leq (b \rightarrow c)\}.\]
\end{defi} 
\begin{theo}\label{theo:miquel}\ 
\begin{enumerate}
\item If $A$ is a $\qoca$, then for all $a,b \in A$, $a b \leq a
  \,\sharp\, b$.
\item If $A$ is a $\ioca$ with adjunctor $\ce$, we have that for all
  $a,b \in A$: $(\ce a) \,\sharp\, b \leq a b$.
\item If $A$ is a $\foca$ we have that for all
  $a,b \in A$: $ a \,\sharp\, b = a b$.
\end{enumerate}
\end{theo}
\begin{proof}\  
\begin{enumerate}
\item Follows from condition (PA): if $a 
  \leq (b \rightarrow c)$ we have that $a b \leq c$, and then $a
  b \leq a \,\sharp\, b$.
\item From $a b \leq a b$, we obtain that $\ce a
  \leq b \rightarrow (a b)$, and this implies that $(\ce a)\,
  \,\sharp\, b \leq a b$. 
\item Is proved as (2), but without $\ce$.
\end{enumerate} 
\end{proof}

\bigskip
\item \label{item:dfireal} In \cite{kn:ocar}, Section 6, we illustrated
  that we can define realizability for $\ioca$s and thus, to define
  realizability in higher-order arithmetics.  In what follows we
  mention briefly how to adapt the constructions therein, to the
  context of a $\foca$, and hence we show that the ordered combinatory
  algebras considered above, can be taken as an adequate platform
  to do realizability. 

The adaptations are minor: the kinds, the language
  and the type system are the same. We have to sligthly modify the
  typing rule of the implication that becomes:

\begin{prooftree}
    \AxiomC{$\Gamma, x:A^o\vdash p:B^o$}
    \RightLabel{\scriptsize{$(\to_i)$}} \UnaryInfC{$\Gamma\vdash
      \lambda^*x.p:(A^o\Rightarrow B^o)^o$}
  \end{prooftree}
  
The interpretation of $\mathcal L^\omega$ is defined in the same
manner than in \cite{kn:ocar}, and the adequacy follows also in the
same manner, the only new problem to deal with is the adequacy for the
implication rules. This follows directly using that the $\foca$
satisfies the full adjunction property.  Indeed, for the implication
rule: $(\to)_i$ we assume $\mathcal{A}\models \Gamma, x:A^o\vdash
    p:B^o$ where $\Gamma=x_1:A_1^o, \dots, x_k:A_k^o$. Consider an
    assignement $\mathfrak{a}$ and~$b_1, \dots, b_k\in A$ such that
    $b_i\leq \llbracket A^o_i\rrbracket$. We get:
      \begin{center}
      \begin{tabular}{rcl}
      $(\lambda^*x.p)\{x_1:=b_1, \dots, x_k:=b_k\}\circ \llbracket
        A^o\rrbracket$&$=$&$\lambda^*x.(p\{x_1:=b_1, \dots,
        x_k:=b_k\})\circ \llbracket
        A^o\rrbracket\leq$\\ &&$p\{x_1:=b_1, \dots, x_k:=b_k,
        x:=\llbracket A^o\rrbracket\}\leq$\\ &&$\llbracket
        B^o\rrbracket$
      \end{tabular}
      \end{center}
      for the first inequality see Paragraph \ref{item:programing},
      the last inequality follows directly from the assumption that
      $\mathcal{A}\models \Gamma, x:A^o\vdash p:B^o$. Applying the
      full adjunction relation between $\circ$ and $\rightarrow$ we
      deduce that $\lambda^*x.p\{x_1:=b_1, \dots, x_k:=b_k\}\leq
      \llbracket (A^o\Rightarrow B^o)^o\rrbracket$. As this is proved
      for all assignements $\mathfrak{a}$, we conclude that
      $\mathcal{A}\models \Gamma\vdash \lambda^*x.p:(A^o\Rightarrow
      B^o)^o$.
\section{Operations in sets of terms and stacks}\label{section:aksrl}
\item 
We briefly recapitulate the definitions and the basic operations in an
\emph{Abstract Krivine Structure} a.k.a. $\mathcal {AKS}$ and
introduce new operations.  Compare with \cite{kn:ocar} where some
aspects of the work of J.L. Krivine and T. Streicher are reformulated
(see \cite{kn:kr2008} and \cite{kn:streicher} respectively).

\item\emph{Polarities and closure
  operations.}\label{item:polarclosure} A triple of sets 
  $(\Lambda,\Pi,\Perp)$ with $\Perp \subseteq\Lambda \times \Pi$ being
  a subset or \emph{a relation} (see \cite[Chapter V, Section
    7]{kn:birkhoff}), induces a \emph{polarity} as illustrated below.
  A generic triple as above will be denoted as $\mathcal
  R=(\Lambda,\Pi,\Perp)$.

The elements of $\Lambda$ and of
$\Pi$ are called respectively \emph{terms} and \emph{stacks} and the
elements of $\Lambda \times \Pi$ are called \emph{processes}. The
processes are pairs $(t,\pi)$, usually denoted by $t \star \pi$ (c.f.:
\cite{{kn:kr2003}}).
If $t \star \pi \in \Perp$, we write $t \perp \pi$, and say
that ``$t$ is orthogonal to $\pi$'' or that ``$t$ realizes $\pi$''. If
$P \subseteq \Pi$ and $t \perp \pi$ for all $\pi \in P$ we say that
``$t$ realizes $P$'' and write $t \perp P$. Simetrically if $t \star
\pi \in \Perp$, we say 
that ``$\pi$ is orthogonal to $t$''. If
$L \subseteq \Lambda$ and $t \perp \pi$ for all $t \in L$ we say that
$\pi$ is perpendicular to $L$. 

%\begin{enumerate}
%\item 
\begin{defi}\label{defi:erre} Given a triple $\mathcal R$ as
  above, we define the 
  following order reversing maps and families of sets:
  \begin{align*} (\quad)^{\perp}:\mathcal
    P(\Lambda)&\xrightarrow{\hspace*{0.5cm}} \mathcal P(\Pi)\\ L
    &\longmapsto  L^{\perp}=\{\pi \in \Pi|\,\, \forall t \in
    L, t \perp \pi\};
  \end{align*}
  \begin{align*} \hspace*{-.1cm}{}^{\perp}(\quad):\mathcal
    P(\Pi)&\xrightarrow{\hspace*{.5cm}} \mathcal P(\Lambda)\\ P
    &\longmapsto {}^{\perp}P=\{t \in \Lambda|\,\, \forall \pi \in P, t
    \perp \pi\}.
  \end{align*}
  The maps are called the polar maps and the subsets $L^\perp$ and
  ${}^\perp P$ are called the polars --or perpendiculars-- of $L$ and
  $P$ respectively. 
\end{defi}
\begin{obse}\label{obse:props-perps}
The maps $(\quad)^{\perp}:\mathcal
P(\Lambda)\rightarrow \mathcal P(\Pi)$ and
${}^{\perp}(\quad):\mathcal P(\Pi)\rightarrow \mathcal P(\Lambda)$ are
order reversing. 

For an arbitrary $L \in \mathcal P(\Lambda)$ and $P \in \mathcal
P(\Pi)$, one has that ${}^{\perp}( L^{\perp}) \supseteq L$ and
$({}^{\perp}P)^{\perp} \supseteq P$. 

For an arbitrary $L \in
\mathcal P(\Lambda)$ and $P \in \mathcal P(\Pi)$, one has that
$({}^\perp(L^{\perp}))^{\perp} = L^{\perp}$ and
${}^{\perp}(({}^\perp P)^\perp) = {}^{\perp}P$.
\end{obse}
%\item 
The roles of $\Lambda$ and $\Pi$ up to now, are completely
symmetric. For reasons that will be clear later, for the next
considerations we work with $\Pi$, i.e. with subsets of stacks,
noting that the corresponding definitions for $\Lambda$ are
identical.
%\item 
\begin{defi}
The set $(({}^\perp P)^{\perp})$ is abreviated as $\overline
{P}$ (when $P$ is a wide expression, we write $(P)^{-}$ instead of
$\overline{P}$). 

The map $P\mapsto (({}^\perp
P)^{\perp})=\overline{P}$ is called the \emph{closure 
operation} associated to the polarity. The family of closed subsets
associated to the closure operation (see \cite[Chapter V, Section 7,
  Theorem 19]{kn:birkhoff}) is denoted as:
\[\mathcal P_{\perp}(\Pi)=\{P \subseteq \Pi:\, \overline{P}=
P\}.\]
\end{defi}
\begin{obse}
\label{item:basicpropincl} The basic properties of the inclusion
$\mathcal P_\perp(\Pi) \subseteq \mathcal P(\Pi)$ can be expressed
by saying that the first is a reflective subcategory of the
second. The map $(P \mapsto \overline{P}):
\mathcal P(\Pi) \rightarrow \mathcal P_\perp(\Pi)$ is the reflector
(see \cite{kn:bor1}). The reflective property can be expressed as:
for $P \in \mathcal P(\Pi)$ and $Q \in \mathcal P_\perp(\Pi)$,
then:
\[P \subseteq Q \Leftrightarrow \overline{P}
\subseteq Q  
\Leftrightarrow {}^{\perp}Q \subseteq {}^{\perp}P.\]
Notice that the polar maps are antitone bijections between
$\mathcal{P}_\perp(\Lambda)$ and $\mathcal{P}_\perp(\Pi)$, each
inverse of the other.  
\end{obse}

We define another closure operator
(not coming in the standard fashion from the polarity) as follows:
\begin{defi}\label{defi:basicpropinc2}
  \[P \mapsto
  \widehat{P}:=\bigcup\{\overline{\pi}: \pi \in P\}: \mathcal
  P(\Pi)\rightarrow \mathcal P(\Pi).\] 
When $P$ is a too wide expression, we write $(P)^\wedge$ instead of
$\widehat{P}$. 

The family of closed subsets
associated to this closure operation is denoted as:
\[\mathcal P_{\bullet}(\Pi)=\{P \subseteq \Pi:\, \widehat{P}=
P\}.\]
\end{defi}
\begin{obse}\label{item:basicpropinc2}
A reflective property also holds for $(\ )^\wedge$\ : 
for $P \in \mathcal P(\Pi)$ and $Q \in \mathcal P_\bullet(\Pi)$,
then:
\[P \subseteq Q \Leftrightarrow \widehat{P} \subseteq Q.\]
\end{obse}
We define now an interior operator  as follows:
\begin{defi}\label{defi:interior}
\[P \mapsto
\widetilde{P}:=\{\pi \in P: \overline{\pi} \subseteq P\}: \mathcal
P(\Pi)\rightarrow \mathcal P(\Pi).\] 
When $P$ is a too wide expression, we write 
$(P)^\wedge$ instead of $\widetilde{P}$. 

Notice --see the proof below--
that we also have: \[\mathcal P_{\bullet}(\Pi)=\{P \subseteq \Pi:\,
\widetilde{P}= P\}.\] 
\end{defi} 
\begin{obse}\label{item:interior}
The corresponding reflection property in this case reads as: 
$P \in \mathcal P_\bullet(\Pi)$ and $Q \subseteq \mathcal P(\Pi)$,
then:
\[P \subseteq Q \Leftrightarrow P \subseteq \widetilde{Q}.\]
\end{obse}
\begin{obse}\label{obse:appendix} 
Let us recall that a \emph{closure operator} in $\Pi$ is a map $P
\mapsto \operatorname{c}(P): \mathcal P(\Pi) \rightarrow \mathcal
P(\Pi)$ such that: 
\begin{itemize}
\item For all $P \subseteq \Pi$, $P \subseteq
\operatorname{c}(P)$.  
\item For all $P \subseteq \Pi$, $\operatorname{c}(\operatorname{c}(P))=
\operatorname{c}(P)$. 
\item For all $P,Q\subseteq \Pi$, $P \subseteq Q$
implies that $\operatorname{c}(P) \subseteq \operatorname{c}(Q)$.
\end{itemize} 
The operator is said to be a \emph{topological closure operator} if:
\begin{itemize}
\item For all $P,Q\subseteq \Pi$, then $\operatorname{c}(P \cup Q) =
\operatorname{c}(P) \cup \operatorname{c}(Q)$.
\item $\operatorname{c}(\emptyset)=\emptyset$
\end{itemize}
It is said to satisfy the
\emph{Alexandroff condition} if: 
\begin{itemize}
\item For every $\{P_i: i \in
I\}\subset\mathcal{P}(\Pi)$ we have $\operatorname{c}(\bigcup\{P_i: i
\in I\})=\bigcup\{\operatorname{c}(P_i): i \in I\}$.  
\end{itemize}
In the above
situation it is customary to define the set of
$\operatorname{c}$--closed subsets of $\Pi$ as: $\mathcal
P_{\operatorname c}(\Pi)=c(\mathcal P(\Pi))=\{P \subseteq \Pi:
\operatorname{c}(P)=P.\}$ Observe that if $\{P_i: i \in
I\}\subset\mathcal{P}(\Pi)$ is a family of $\operatorname{c}$--closed sets, 
the set $\bigcap_i P_i$ is also closed.

The notion of \emph{interior operator} is dual to the one of
\emph{topological closure operator}. The open sets defined by an
\emph{interior operator} $i$ are the sets $P$ such that $i(P)=P$. 
\end{obse}
The operators defined above satisfy the following:
\begin{obse}\label{item:propofperpbullet} 
  The operator $P \mapsto \overline{P}: \mathcal P(\Pi)
  \rightarrow \mathcal P(\Pi)$ is a closure operator $P \mapsto
  \widehat{P}: \mathcal P(\Pi)\rightarrow \mathcal P(\Pi)$ is a
  topological closure operator and $P \mapsto \widetilde{P}:\mathcal
  P(\Pi)\rightarrow \mathcal P(\Pi)$ is an interior
  operator. Moreover, the topology in $\Pi$ associated to the closure
  operator $P \mapsto \widehat{P}$ is an Alexandroff topology
  (c.f.:\ref{obse:topologies-alexandroff}).   
  The proof of this fact follows directly from the
  adjunction properties, of the operators: $P \mapsto \overline{P}, P
  \mapsto \widehat{P}, P \mapsto \widetilde{P}$.  
\end{obse}
\begin{teo}
  \label{item:hatperp=perp}For an arbitrary subset $P \subseteq
  \Pi$:
  \begin{enumerate} 
  \item\label{iitem:hatperp=perp1} $\widetilde{P} \subseteq P
    \subseteq \widehat{P} \subseteq 
    \overline{P}$,  
  \item\label{iitem:hatperp=perp2} ${}^\perp P\ =\ {}^\perp(\widehat{P})$ and 
    $\overline{P}\ =\ \overline{\widehat{P}}$,
  \item\label{iitem:hatperp=perp3} $\mathcal P_{\perp}(\Pi) \subseteq
    \mathcal P_{\bullet}(\Pi) 
    \subseteq\mathcal P(\Pi)$,
  \item\label{iitem:hatperp=perp4} The following are equivalent: 
    \begin{enumerate}
      \item\label{iiitem:hatperp=perp4a} $P \in \mathcal
        P_\bullet(\Pi)$
      \item \label{iiitem:hatperp=perp4b}$\pi \in P \Leftrightarrow
        \overline{\pi} 
        \subseteq P$.  
      \item\label{iiitem:hatperp=perp4c} $\widetilde{P}=P$
    \end{enumerate}
  \item\label{iitem:hatperp=perp5} For $P \subseteq \Pi$, then $P
    \in \mathcal P_\perp(\Pi) \, 
    \Leftrightarrow \, \widetilde{P}=P=\widehat{P}=\overline{P}$. 
  \end{enumerate}
\end{teo}
\begin{proof}\ 
  \begin{enumerate}
    \item Firstly observe that ${}^\perp P=\bigcap \{{}^\perp\{\pi\}:
      \pi \in P\}$. 

      Thus  
      $\overline{P}= ({}^\perp P)^\perp=(\bigcap \{{}^\perp\{\pi\}: \pi
      \in P\})^\perp \supseteq \bigcup \{\overline{\pi}: \pi \in
      P\}=\widehat{P}$. The rest of the inclusions are straightforward.
    \item It follows directly from (\ref{iitem:hatperp=perp1}) that
      ${}^\perp P\supseteq 
      {}^\perp(\widehat{P})$. For the reverse inclusion, observe that
      $t\perp \pi$ if 
      and only if $t\perp \overline{\pi}$ (c.f.:
      \ref{obse:props-perps}). The other equality is obtained applying
      the corresponding polar map to ${}^\perp P = {}^\perp (\widehat{P})$. 
   \item Let us consider $P\in\mathcal{P}_\perp(\Pi)$, i.e.:
     $\overline{P}=P$. By (\ref{iitem:hatperp=perp1}) we get 
     $P\subseteq \widehat{P}\subseteq \overline{P}=P$, which implies
     that $P\in\mathcal{P}_\bullet(\Pi)$, thus proving the left
     inclusion. The other one is evident.  
    \item The equivalence of (\ref{iiitem:hatperp=perp4a}) and
      (\ref{iiitem:hatperp=perp4b}) 
      follows directly from the 
      definition of $\widehat{P}$. The equivalence of
      (\ref{iiitem:hatperp=perp4b}) and (\ref{iiitem:hatperp=perp4c})  
      follows from the definition of $\widetilde{P}$.  
    \item This equivalence follows directly from
      (\ref{iitem:hatperp=perp3}) and (\ref{iitem:hatperp=perp4}). 
    \end{enumerate}
  \end{proof}
%\end{enumerate} 

\begin{obse}\label{obse:topologies-alexandroff}
By the equivalence of
(\ref{item:hatperp=perp})(\ref{iiitem:hatperp=perp4a}) and 
(\ref{item:hatperp=perp})(\ref{iiitem:hatperp=perp4c}), we have that the
open sets induced by 
$(\ )^\wedge$ are just the closed sets induced by $(\ )^\wedge$,
thus concluding that both topologies are Alexandroff.   
\end{obse}
\noindent

%\end{enumerate}

\item\emph{Order and Completeness}. \label{item:completeness}\noindent
  If $\mathcal R$ is as above, all the three sets: $\mathcal
  P_{\perp}(\Pi) \subseteq \mathcal P_{\bullet}(\Pi) \subseteq\mathcal
  P(\Pi)$ are complete with the order $\subseteq$.  Indeed, the
  functors $(-)^{\wedge}:\mathcal P(\Pi) \rightarrow \mathcal
  P_\bullet(\Pi)$ and $(-)^{-}:\mathcal P(\Pi) \rightarrow \mathcal
  P_\perp(\Pi)$ are retractions of $\mathcal P(\Pi)$ into $\mathcal
  P_\bullet(\Pi)$ and $\mathcal P_\perp(\Pi)$ respectively and the
  considerations that follow about the form of the product and
  coproduct (or $\operatorname{inf}$ and $\operatorname{sup}$) in
  $\mathcal P_\bullet(\Pi)$ and $\mathcal P_\perp(\Pi)$ are 
  particular cases of general results about the construction of limits
  and colimits in reflexive subcategories --see \cite{kn:bor1} and
  \cite{kn:bor2} pg. 118 and pg. 196 respectively--.
 \begin{enumerate}
\item It is clear that for $\mathcal X \subseteq \mathcal P(\Pi)$,
  $\operatorname{sup}(\mathcal{X})=\bigcup\mathcal{X}\,,
  \,\operatorname{inf}(\mathcal{X})= \bigcap\mathcal{X}$.
\item For $\mathcal X \subseteq \mathcal P_\bullet(\Pi)$,
  $\operatorname{sup}_\bullet (\mathcal X)=\bigcup \mathcal{X}\,
  ,\,\operatorname{inf}_\bullet(\mathcal X)=\bigcap \mathcal{X}$.  The
  property for the $\inf_\bullet$ follows by general considerations,
  the one concerning $\sup_\bullet$ follows from Alexandroff condition. 

\item For $\mathcal X \subseteq \mathcal P_\perp(\Pi)$,
  $\operatorname{sup}_\perp(\mathcal X)=\overline{\bigcup
  \mathcal{X}}\,,\,\operatorname{inf}_\perp(\mathcal X)=\bigcap
  \mathcal{X}$.
This again follows from the general facts mentioned before.

\end{enumerate}  

\item\begin{example}\label{item:nonclosed} Consider the triple
  $(\Lambda,\Pi,\Perp)$ 
  with $\Lambda=\Pi=\mathbb R^3$, and $\Perp$ being the usual perpendicularity
  relation.
%Moreover
 % $\{e_1,e_2,e_3\}$ will be a positive orthonormal basis.
\noindent
In this context for $P \subseteq \mathbb R^3$, $\overline{P}$ is the
linear subspace generated by $P$ and $\widehat{P}$ is the cone with
vertex in the origin generated by $P$. Then $\mathcal P_\perp(\mathbb
R^3)$ is the set of linear subspaces and $\mathcal P_\bullet(\mathbb
R^3)$ is the set of cones. 
\bigskip
\end{example}
\item\label{item:AKS} We give one more step heading towards the definition due
  to Streicher in \cite{kn:streicher} of an $\mathcal{AKS}$.
\begin{defi}\emph{Realizability lattice}.\label{def:totalaks}

A quadruple $\Big(\Lambda,\Pi,\Perp,\operatorname{push}\Big)$
consisting of a pair of sets $\Lambda$ --set of terms-- $\Pi$ --set of
stacks--, a subset $\Perp \subseteq \Lambda \times \Pi$ and a function
$\operatorname{push}: \Lambda \times \Pi \rightarrow \Pi$ is called a
\emph{realizability lattice}.
\medskip

\noindent \emph{Notation}: The family of the realizability lattices
is denoted as $\mathcal{RL}$ and one of its generic
elements will be called $\mathcal R$.

We use also the notation $\operatorname{push}(t,\pi)=t
{\leadsto}\pi=t \cdot \pi$ for $t \in \Lambda$ and $\pi \in \Pi$.
\end{defi}

\item \label{defi:operations}\emph{The operations of application and
  implication: first approximation}.
\noindent
Below we define three pairs of operations.  As the role of the first
and third operations in the formalization by Streicher of
realizability theory (see \cite{kn:streicher}) has been considered in
\cite[Section 2]{kn:ocar}, we concentrate our attention in the new set
$\mathcal P_\bullet(\Pi)$ and the new operations $\ast_\bullet$ and
$\leadsto_\bullet$. The first and third are mentioned only for
comparison reasons.
\begin{defi}\label{item:adempty} 
For $L \in \mathcal P(\Lambda)\,,\,P \in \mathcal P(\Pi)$ we define
the following subsets, given by the \emph{push} operations:
\[
  \begin{array}{lcccllll}
  \leadsto\,\,&:&\!\!\!\mathcal P(\Lambda) \times \mathcal P(\Pi)
  \rightarrow \mathcal P(\Pi)&,& L{\leadsto}P&:=&\{t\cdot\pi: t \in L,
  \pi \in P\}
  &\,\,;\\ 
  \leadsto_\bullet&:&\!\!\!\mathcal P(\Lambda) \times \mathcal
  P(\Pi) \rightarrow \mathcal P_\bullet(\Pi)&,&
  L{\leadsto_\bullet}P&:=&(L \leadsto
  P)^\wedge&\,\,;\\ 
  \leadsto_\perp&:&\!\!\!\mathcal P(\Lambda) 
  \times \mathcal P(\Pi)\rightarrow \mathcal P_\perp(\Pi)&,& 
  L{\leadsto_\perp}P&:=&(L \leadsto P)^{-}
  &
  \end{array}
\]    
For $L \in \mathcal P(\Lambda)\,,\,P \in
  \mathcal P(\Pi)$ we define the following \emph{conductors} of $L$
  into $P$:
\[
  \begin{array}{lccccclll}
   \ast\,\,&:&\!\!\!\mathcal P(\Pi)
  \times \mathcal P(\Lambda) \rightarrow \mathcal P(\Pi)&,&P \ast
  L&:=&\{\pi \in \Pi: L\leadsto \pi \subseteq P\}&=&\{\pi \in 
  \Pi: L\cdot\pi \subseteq P\}\,\,;\\ 
  \ast_\bullet&:&\!\!\!\mathcal P(\Pi) \times \mathcal P(\Lambda)
  \rightarrow \mathcal P_\bullet(\Pi)&,&
  P\ast_\bullet L&:=&\{\pi \in \Pi: L\leadsto \overline{\pi} \subseteq
  P\}&=&\{\pi \in \Pi: L\cdot\overline{\pi} \subseteq P\}\,\,;\\ 
  \ast_\perp&:&\!\!\!\mathcal P(\Pi) \times \mathcal P(\Lambda)
  \rightarrow \mathcal P_\perp(\Pi)&,&
  P \ast_\perp L&:=&\{\pi \in
  \Pi: L\leadsto \pi \subseteq P\}^{-}&=&(P \ast L)^{-} 
  \end{array}
\]    
\end{defi}
\noindent

\begin{obse}\label{obse:redefinition-ast-bullet} By definition of $\ast$,
  $L.\overline{\pi}\subseteq P$ if 
  and only if $\overline{\pi}\subseteq P\ast L$. We conclude
  that $P\ast_\bullet L=(P\ast L)\widetilde{\ }$. 
 
Notice a basic difference in the definition of the pair of 
operations  $(\ast_\perp,\leadsto_\perp)$ and the pair
$(\ast_\bullet,\leadsto_\bullet)$. In the first row the operations
are given in terms of the closure $(-)^-$ while in the second row by
the \emph{interior} operator $(-)\,\,\widetilde{}\,\,$  and the
closure $(-)^\wedge$. 
\begin{center}
\begin{tabular}{c|c}
 $\emph{Conductor}$& $\emph{Push}$\\ 
  \hline &\\ 
  $P \ast_\perp L\quad =\quad (P\ast L)^-$&$L{\leadsto_\perp}P\quad = \quad 
  (L{\leadsto}P)^-$ \\ 
  \hline &\\
  $P\ast_\bullet L\quad =\quad (P \ast L)\,\,\widetilde{}$&$L{\leadsto_\bullet}P
  \quad =\quad (L{\leadsto}P)^\wedge$
\end{tabular}
\end{center}
\end{obse}
\begin{obse}\label{obse:table}\ 
\begin{enumerate} 
\item \label{item:table2} 
Applying (\ref{obse:redefinition-ast-bullet}) and
(\ref{item:hatperp=perp})(\ref{iitem:hatperp=perp1}), we get the
following table:  
\medskip
\begin{center}
\begin{tabular}{  c| c }
 $\emph{Conductor}$& $\emph{Push}$\\ 
  \hline &\\ 
  $P \ast_\bullet L
  \quad \subseteq\quad  P \ast L \quad\subseteq\quad P \ast_\perp
  L$&$L \leadsto P 
  \quad \subseteq\quad L \leadsto_\bullet P \quad \subseteq\quad L
  \leadsto_\perp P$ \\ 
\end{tabular}
\end{center}
\medskip
\medskip
%% xxx \ref{item:completeness}, Observation \ref{obse:raro} and the rest
%% follows from the considerations of \ref{item:polarclosure}, item
%% }\eqref{item:propofperpbullet}. 
\item \label{item:apraro} The reflection properties of the closures
  and of the interior operator --see Paragraph 
  \ref{item:polarclosure}, observations \eqref{item:basicpropincl},
  \eqref{item:basicpropinc2} and \eqref{item:interior}-- leads to the
  following tables:
\medskip
\begin{enumerate}
\item If $R\in \mathcal{P}_\perp(\Pi)$, $L\in\mathcal{P}(\Lambda)$ and
$P\in\mathcal{P}(\Pi)$:
\medskip
\medskip 

\begin{tabular}{c|c}
 $\emph{Conductor}$& $\emph{Push}$\\ 
  \hline &\\ 
  $P \ast L \subseteq
  R\quad \Leftrightarrow\quad P \ast_\perp L \subseteq R$ & $L {\leadsto} P
  \subseteq R\quad \Leftrightarrow\quad L {\leadsto_\perp} P \subseteq R$ \\ 
\end{tabular}
\medskip
\medskip
\medskip

\item If $R\in\mathcal{P}_\bullet(\Pi)$, $L\in\mathcal{P}(\Lambda)$ and
$P\in\mathcal{P}(\Pi)$:
\medskip
\medskip

\begin{tabular}{c|c}
 $\emph{Conductor}$& $\emph{Push}$\\ 
  \hline &\\ 
  $R \subseteq P \ast L \quad\Leftrightarrow\quad 
  R \subseteq P \ast_\bullet L $&$L {\leadsto} P \subseteq R
  \quad \Leftrightarrow\quad  L {\leadsto_\bullet} P \subseteq R$\\
\end{tabular}
\end{enumerate}
  %% Indeed, considering $\Phi_{L,Q}(x):=L.x\subseteq Q$, we have that
  %% $\Pi_{\Phi_{L,Q}}=Q\ast L$ and
  %% $\Pi_{(\Phi_{L,Q})_\bullet}=Q\ast_\bullet L$. Applying Observation
  %% \ref{obse:raro}, \eqref{item:rarito} we obtain that $R\subseteq
  %% Q\ast L$ implies $R\subset Q\ast_\bullet L$. The converse
  %% implication is a direct consequence of \eqref{item:table}.
\end{enumerate}
\end{obse}
%%%%%%%%%%%%%%%%%%%%%%%%%%%%%%%%%%%%%%%%%%%%%%%%%%%%%%%%%%%%%%%%%%%%%%%%%%%%%%% 
\item \emph{The adjunction of the new operators: first
  approximation}.\label{item:conductor} We have the following
  adjunction result for ${\mathcal R}$, a realizability lattice.
\begin{theo}\label{theo:adjunctionbullet}
The maps $\ast_\bullet: {\mathcal P}(\Pi) \times {\mathcal P}(\Lambda)
\rightarrow {\mathcal P}_\bullet(\Pi)\,,\,\leadsto_\bullet: {\mathcal
P}(\Lambda) \times {\mathcal P}(\Pi) \rightarrow {\mathcal P}_\bullet(\Pi)$
satisfy the \emph{adjunction property}: If $P,R \in {\mathcal
P}_\bullet(\Pi)$ and $L \in {\mathcal P}(\Lambda)$, then $L
\leadsto_\bullet R \subseteq P$ if and only if $R \subseteq P
\ast_\bullet L$.
\end{theo}
\begin{proof}
In the above situation $L \leadsto _\bullet R \subseteq P$, if
  and only if $L \leadsto R \subseteq P$ if and only if $R \subseteq P
  \ast L$ if and only if $R \subseteq P \ast_\bullet L$ --see
  Observation \ref{obse:table}, \eqref{item:apraro}--.
\end{proof}

\item \label{item:hap} Concerning the adjunction property for the
  other products the following holds as it is proved in \cite[Section
    2]{kn:ocar}.
\begin{enumerate} 
\item The maps $\ast : {\mathcal P}(\Pi) \times {\mathcal P}(\Lambda)
  \rightarrow {\mathcal P}(\Pi)$ and $\leadsto : {\mathcal P}(\Lambda)
  \times {\mathcal P}(\Pi) \rightarrow {\mathcal P}(\Pi)$ , satisfy the
  \emph{full adjunction property}: $L \leadsto R \subseteq P$ if and
  only if $R \subseteq P \ast L$.

\item \begin{enumerate}
\item The maps $\ast_\perp : {\mathcal P}(\Pi) \times {\mathcal
  P}(\Lambda) \rightarrow {\mathcal P}_\perp(\Pi)$ and $\leadsto_\perp :
  {\mathcal P}(\Lambda) \times {\mathcal P}(\Pi) \rightarrow {\mathcal
  P}_\perp(\Pi)$ satisfy the \emph{half adjunction property}: $P,R \in
  {\mathcal P}_\perp(\Pi)$ and $L \in {\mathcal P}(\Lambda)$, and $L
  \leadsto_\perp R \subseteq P$, then $R \subseteq P \ast_\perp L$.
\item The other \emph{half} of the adjunction property can be
  partially recuperated using the combinator $\cE$ as follows --see
  \cite[Section 2,\textsection 2.3, Theorem 2.13]{kn:ocar} (and for
  another proof see Theorem \ref{theo:fulladj} below):
\[\text{If}\quad
  R\quad \subseteq\quad P \ast_\perp L \quad\mbox{then}\quad L
  \leadsto_\perp R\quad \subseteq\quad 
  \{\cE\}^\perp \ast_\perp {}^\perp P.\]
\end{enumerate}
\end{enumerate}
\bigskip
\item\label{item:nonclosed2}
We can complete the Example presented in Paragraph
\ref{item:nonclosed} adding: $\operatorname{push}(v,w)=\langle v,w
\rangle v \wedge w$ --if $v,w \in \mathbb R^3$ where we are denoting
as $\langle v,w \rangle\,,\, v \wedge w \in \mathbb R^3$ the usual
inner and vector product respectively.  Let $\{e_1,e_2,e_3\}$ be a
positive orthonormal basis.  
\begin{enumerate}
\item We construct an example where $P \in {\mathcal P}_\bullet(\Pi)$,
  $L \in {\mathcal P}_\bullet(\Lambda)$ and $P \ast L = P \ast_\bullet
  L \neq P \ast_\perp L$.  Consider $L=\mathbb Re_2$ and $P=\mathbb
  Re_1$. Then: $P \ast L=\{w \in \mathbb R^3:\langle e_2, w\rangle
  e_2 \wedge w \in \mathbb R e_1\}$. If $\langle e_2, w\rangle=0$,
  then $w \in \mathbb R e_1 + \mathbb R e_3$, otherwise $e_2 \wedge w
  \in \mathbb R e_1$ if and only if $w \in \mathbb R e_2 + \mathbb R
  e_3$. Then, $P \ast L=P \ast_\bullet L= (\mathbb R e_1 + \mathbb R
  e_3) \cup (\mathbb R e_2 + \mathbb R e_3) \neq P \ast_\perp L =
  \mathbb R^3$.
\item
Moreover, the full adjunction property does not hold in this situation
for the operators $\leadsto_\perp$ and $\ast_\perp$. Indeed, we have
that if we take $P,L$ as above and $R=\mathbb R(e_1+e_2)$ a direct
computation shows that: $R= \mathbb R(e_1+e_2) \subseteq {P
  \ast_\perp L}=\mathbb R^3$, but $L \leadsto_\perp \mathbb
R(e_1+e_2)=\mathbb Re_3 \not \subseteq \mathbb Re_1=P$.
\item Next we give an example where $L \in {\mathcal P}_\perp(\Lambda)$ and $P
\in {\mathcal P}_\perp(\Pi)$ but $P \ast_\bullet L$ and $P \ast L$
%and $L \bullet P$  
are different.

\noindent  
In the same context than above we take
$\operatorname{push}(v,w)=\langle v,w-w_0 \rangle v \wedge w$ where
$w_0\notin \mathbb R e_1 + \mathbb R e_3$, also consider $L=\mathbb
Re_2$ and $P=\mathbb Re_1$. Then, $P \ast_\bullet L= \mathbb R e_2 +
\mathbb R e_3$, $P \ast L= (\mathbb R e_1 + \mathbb R e_3 + w_0) \cup
(\mathbb R e_2 + \mathbb R e_3)$.
% and $L \bullet P= \mathbb R^3$. 
\end{enumerate}  
\section{Abstract Krivine structures: Operations and combinators}
\label{section:opcomb} 
\noindent
\item\label{item:calculus}  
Recall the definition of an \emph{Abstract Krivine Structure}
a.k.a. ${\mathcal {AKS}}$--see\cite{kn:ocar} where parts of the work of
J.L. Krivine and T. Streicher in \cite{kn:kr2008} and
\cite{kn:streicher} are reformulated.
\begin{defi}\label{defi:aks}Let ${\mathcal R}=(\Lambda,\Pi,\Perp,
  \operatorname{push})$ be an ${\mathcal {RL}}$, and assume that we have
  the following additional elements:
   \begin{enumerate}
\item A function $\operatorname{app}: \Lambda \times \Lambda
  \rightarrow \Lambda$ (denoted also as $\operatorname{app}$ when
  viewed as a map of type $\operatorname{app}:{\mathcal P}(\Lambda)
  \times {\mathcal P}(\Lambda) \rightarrow {\mathcal P}(\Lambda)$). We
  also write $\operatorname{app}(t,s)=ts$ for $t,s \in \Lambda$. We
  use the convention that associates to the left when operating with
  more than three elements,
  \item A set $\mathrm {QP} \subseteq \Lambda$ of ``quasi--proofs'',
    which is closed under application,
\item \label{item:reduction} A fixed pair of elements $\cK,\cS \in
  \mathrm{QP} \subseteq \Lambda$, satisfying the following:
\begin{enumerate}
\item For all $t,s \in \Lambda$ if $t \perp s\cdot\pi$, then $ts \perp
  \pi$;
\item For all $t \perp \pi$ and for all $s \in \Lambda$ we have that
  $\cK \perp t\cdot s\cdot\pi$;
\item If $tu(su) \perp \pi$ for $t,s,u \in \Lambda$, then $\cS \perp
  t\cdot s\cdot u\cdot\pi$.
\end{enumerate}
\end{enumerate}
The structure thus obtained is called an \emph{Abstract Krivine
  Structure} --abbreviated as ${\mathcal {AKS}}$-- and usually ${\mathcal
K}$ will denote a generic ${\mathcal {AKS}}$\footnote{\emph{Strictu
    sensu}, the standard definition of ${\mathcal {AKS}}$ includes
  additional structure --see \cite{kn:report} or
  \cite{kn:streicher}--. We only take the basic elements we need for
  our present considerations.}.
\end{defi}
\item\label{item:AKS2} We introduce a new operation ${\mathcal P}(\Pi)
  \times {\mathcal P}(\Lambda) \rightarrow {\mathcal P}(\Pi)$ that allows
  to express in terms of this operation the basic reduction rule of
  an ${\mathcal {AKS}}$: ``\emph{If $t \perp s\cdot\pi$, then $ts \perp
    \pi$}'' as well as the axioms for $\cK$ and $\cS$.

\begin{defi} We define the operation $\clubsuit: {\mathcal P}(\Pi)
  \times {\mathcal P}(\Lambda) \rightarrow {\mathcal{P}}_{\perp}(\Pi)$
  as: $P \,\clubsuit\, L:= ({}^\perp P L)^\perp$.
\end{defi} 
Observe that $P\,\clubsuit\, L$ is in ${\mathcal{P}}_{\perp}(\Pi)$ because
it is the perpendicular of a subset of $\Lambda$.

This operation is really the ``transfer'' of the operation
$\operatorname{app}$ from ${\mathcal P}(\Lambda)$ to ${\mathcal P}(\Pi)$
through the perpendicularity map, as the diagram below ilustrates:
\[\xymatrix
{{\mathcal P}(\Pi) \times {\mathcal P}(\Lambda)\ar[r]^-{\clubsuitLittle}
  \ar[d]_{{}^\perp(-) \times \operatorname{id}}&{\mathcal
  P}_\perp(\Pi)\\ {\mathcal P}(\Lambda)\times{\mathcal
  P}(\Lambda)\ar[r]_-{\operatorname{app}}&{\mathcal
  P}(\Lambda)\ar[u]_{(-)^\perp} }
\]
\begin{lema}\label{lema:s1} The following statements are equivalent:
  \begin{enumerate} 
  \item For all $t,s \in \Lambda,\, \pi \in \Pi$: if $t \perp s\cdot\pi$
    then $ts \perp \pi$.

  \item $P \ast L \subseteq P \,\clubsuit\, L$ for all $L \in
    {\mathcal P}(\Lambda), P \in {\mathcal P}(\Pi)$.

  \item $P \ast_\perp L \subseteq P \,\clubsuit\, L$ for all $L \in
    {\mathcal P}(\Lambda), P \in {\mathcal P}(\Pi)$.
  \item \label{item:s14} For $P \subseteq \Pi$ and $L \subseteq
    \Lambda$\emph{:} if $t \perp P$ and $\ell \in L$ then $t\ell \perp
    P \ast_\bullet L \subseteq P \ast L$.
\end{enumerate}
\end{lema}
\begin{proof}\    
\begin{itemize} 
  \item To prove that (1) implies (2) take $\pi \in P \ast L$, $t\perp P$
  and $\ell\in L$. Since $\pi\in P\ast L$ then~$\ell\cdot\pi\in P$, then
  $t\perp \ell\cdot\pi$. Applying (1) we get $t\ell\perp\pi$ and hence
  $\pi\in P \,\clubsuit\, L$.
 
  To prove that (2) implies (1), just take $L:=\{s\}$ and
  $P:=\{t\}^\perp$.
\item Statements (2) and (3) are equivalent because $P \clubsuit L \in
  {\mathcal P}_\perp(\Pi)$ and $\overline{(-)}$ is a reflection.
  \item Condition (4) can be formulated as the inclusion: ${}^\perp PL
    \subseteq {}^\perp(P \ast L)$ that is clearly equivalent to
    condition (3).
\end{itemize}
\end{proof}
\begin{obse}
In particular, in an $\mathcal{AKS}$ we have the following chain of
inequalities (see Observation \ref{obse:table}): For all $P
\subseteq \Pi$ and $L \subseteq \Lambda$,
\begin{equation}\label{eqn:inequalities2} P \ast_\bullet L \quad
  \subseteq\quad  P
  \ast L \quad \subseteq\quad P \ast_\perp L 
  \quad \subseteq\quad P \,\clubsuit\, L.
\end{equation} 
Notice that the operation $\Box$, can be defined only when we furnish
the original $\mathcal{RL}$ with the necessary additional stucture
necessary to produce an $\mathcal{AKS}$ --i.e. with: the application
map; the set of quasi proofs with its distinguished elements, and the
axioms relating the application with the push through
perpendicularity, that also characterize the distingushed
elements--. This $\Box$--operation is an upper bound for the
operations $\ast, \ast_\bullet$ and $\ast_\perp$, which are more
elementary and defined in the realm of the underlying $\mathcal{RL}$.
\end{obse}
\bigskip
\item \label{item:morecombinators} Next we recall some constructions
  and results based upon the properties of the basic combinators $\cK,
  \cS$. Most of this appears in \cite[Section 2,\textsection 2.3,
    Lemma 2.12]{kn:ocar}. The introduction of the new operations
  allows some additional precision in the inequalities obtained.
\begin{enumerate}
\item \label{item:newcomb} Define for an ${\mathcal {AKS}}$ called
  ${\mathcal K}$ the following elements of $\mathrm{QP} \subseteq \Lambda$:
\begin{enumerate}
\item $\cI:=\cS\cK\cK$,
\item $\cE:=\cS(\cK(\cS\cK\cK))=\cS(\cK \cI)$.
\item $\cB:=\cS(\cK\cS)\cK$. 
\end{enumerate}
\item \label{item:newcombprop} The elements given above satisfy the
  following conditions for all $t, s, \ell \in \Lambda$, $\pi \in
  \Pi$:
\begin{enumerate}
\item $t \perp \pi$ implies that: $\cI \perp t\cdot \pi$,
\item $t\ell \perp \pi$ implies that: $\cE \perp t\cdot\ell\cdot \pi$.
\item $t\perp (s\ell)\cdot \pi$ implies that: $\cB \perp t\cdot s\cdot
  \ell \cdot \pi$.  
\end{enumerate}
\item The following results hold: 
\begin{enumerate}
\item For all $P \in {\mathcal P}(\Pi),L \in {\mathcal P}(\Lambda)$, then
  ${\cK\!} \perp {}^\perp P \leadsto (L \leadsto P)$.
\item For all $P,Q,R \in {\mathcal P}(\Pi)$, then ${\cS\!}
  \perp {}^\perp P \leadsto \left({}^\perp Q \leadsto \Big({}^\perp R
  \leadsto \big((P \ast {}^\perp R) \ast {}^\perp(Q \ast {}^\perp
  R)\big)\Big)\right)$.
\item \label{item:peprima} For all $P \in {\mathcal P}(\Pi), L \in
  {\mathcal P}(\Lambda)$:
${\cE\!} \perp {}^\perp P \leadsto (L
  \leadsto (P\, \clubsuit\, L))$.
\end{enumerate}
\end{enumerate}
\noindent
As mentioned above the validity of (3): (a), (b) and (c) is proved in
\cite[Section 2,\textsection 2.3, Lemma 2.12]{kn:ocar}.
\begin{lema}\label{lema:peprima}
For all $P \in {\mathcal P}(\Pi), L \in {\mathcal P}(\Lambda)$\emph{:}
\begin{enumerate}
\item ${\cE\!}{\cE\!} \perp {}^\perp P \leadsto_\bullet (L
  \leadsto_\bullet (P\, \clubsuit\, L))$.
\item \label{item:peprimamain} $P\, \clubsuit\, L\quad \subseteq
  \quad (\{\cE\cE\!\}^\perp \ast_\bullet {}^\perp P)\ast_\bullet L
  \quad \subseteq\quad \left(\cE\cE ({}^\perp P)\right)^\perp \ast_\bullet L$.
\end{enumerate}
\end{lema}
\begin{proof}\  
\begin{enumerate}
\item Take $t \perp P$, $\ell \in L$ and $\pi \in ({}^\perp P
  L)^\perp= P\, \clubsuit\, L$, hence $\t \ell \perp \pi$ and it
  follows from Paragraph
  \ref{item:morecombinators}. \eqref{item:newcombprop} that $\cE \perp
  t\cdot \ell \cdot \pi$ that implies that $\cE t \perp \ell \cdot
  \pi$ or in other words that $\cE t \in {}^\perp(L \cdot (P \clubsuit
  L))$.

Hence $\cE t \in {}^\perp(L \cdot (P \clubsuit L))= {}^\perp(L
\leadsto (P \clubsuit L))={}^\perp((L \leadsto (P \clubsuit
L))\,\widehat{}\,\,)= {}^\perp(L \leadsto_\bullet (P \clubsuit L))$
where the second equality comes from the statement in Paragraph
\ref{item:polarclosure}.\eqref{item:hatperp=perp} and the others are
simply the definitions.

Stating this assertion as $\cE t \perp \rho$ for $t \perp P$ and $\rho
\in L \leadsto_\bullet (P \clubsuit L)$, and reasoning as before, we
obtain that $\cE \perp \cE\cdot t \cdot \rho$ and $\cE\cE \perp t\cdot
\rho$.

This means that $\cE\cE \in {}^\perp({}^\perp P\cdot( L
\leadsto_\bullet (P \clubsuit L)))= {}^\perp({}^\perp
P\leadsto_\bullet( L \leadsto_\bullet (P \clubsuit L)))$ where the
last equality is obtained in the same manner than above.

\item If we write the perpendicularity relation just proved as:
  $\{\cE\!\cE\!\}^\perp \supseteq {}^\perp P \leadsto_\bullet (L
  \leadsto_\bullet (P\, \clubsuit\, L))$, and then apply twice the
  full adjunction between $\ast_\bullet$ and $\leadsto_\bullet$ we
  obtain $P\, \clubsuit\, L \subseteq (\{\cE\cE\!\}^\perp
  \ast_\bullet {}^\perp P)\ast_\bullet L$. The second inclusion of
  (2) is a consequence of the inclusion of the first into the last
  term in \eqref{eqn:inequalities2} applied to $\{\cE\cE\!\}^\perp
  \ast_\bullet {}^\perp P$.
\end{enumerate}
\end{proof}

\item \label{item:theadjunctor}

\emph{The adjunctor}.  In this paragraph we show that with the help of
the element $\cE$ and the new operations~$\ast_\bullet$ and
$\leadsto_\bullet$ and $\clubsuitLittle$, we can obtain,
similarly than in \cite{kn:ocar} control over the ``other half'' of
the adjunction between the functors $\ast_\perp$ and $\leadsto_\perp$
--compare with \cite[Section 2,\textsection 2.3, Theorem
  2.13]{kn:ocar}.
\begin{defi}\label{defi:pprima} 
Let ${\mathcal K}$ be an ${\mathcal {AKS}}$, define a map $P \mapsto
\eta P: {\mathcal P}(\Pi) \rightarrow {\mathcal P}_\bullet(\Pi)$ by the
formula: $\eta  P= \{\cE\cE\!\}^\perp \ast_\bullet {}^\perp P$.  The
element $\cE$ that is a particular $\eta$--expansor, will be called an
\emph{adjunctor}.
\end{defi}
\begin{obse}\label{obse:futureref}\ 
\begin{enumerate}
\item For future reference we write down the basic inequality
  involving the definition of $\eta P$ --and that follows from equation
  \eqref{eqn:inequalities2}.
\[\begin{split}
\eta P\quad =\quad \{\cE\cE\}^\perp\ast_\bullet {}^\perp P\quad \subseteq\quad
\{\cE\cE\}^\perp\ast\,^\perp P\quad \subseteq\quad 
\{\cE\cE\}^\perp\ast_\perp\,^\perp P  
\quad\subseteq\\
\subseteq\quad \{\cE\cE\}^\perp\clubsuit{}\,^\perp P\quad =\quad 
(\overline{\{\cE\cE\}}({}^\perp P))^\perp\quad \subseteq\quad (\cE\cE (^\perp
P))^\perp.
\end{split}
\]
\item Taking into account that some of the sets in the above
  inequality are elements of ${\mathcal P}_\perp(\Pi)$, taking
  closures we obtain that: \[\overline{\eta P}\quad \subseteq\quad 
  \{\cE\cE\}^\perp\ast_\perp{}^\perp P\quad \subseteq\quad
  \{\cE\cE\}^\perp\clubsuit{}\,^\perp P=
  (\overline{\{\cE\cE\}}({}^\perp P))^\perp\quad \subseteq\quad (\cE\cE (^\perp
  P))^\perp.\]
\end{enumerate}
\end{obse}

The theorem that follows summarizes the inclusion relations obtained
before.
\begin{theo}\label{theo:clubsuitE}
Assume that ${\mathcal K}$ is an ${\mathcal {AKS}}$ and the notations
are as above. Then the following inclusions hold for $P \in {\mathcal
  P}(\Pi)$ and $L \in {\mathcal P}(\Lambda)$:
  \[\begin{split}
  P \ast_\bullet L\quad 
  \subseteq\quad P \ast L\quad \subseteq\quad P \ast_\perp L\quad
  \subseteq\quad P 
  \,\clubsuit\, L \quad\subseteq\\
\subseteq\quad \eta P \ast_\bullet L\quad \subseteq\quad 
  \eta P\ast L 
  \quad\subseteq\quad \eta P \ast_\perp L 
  \end{split}
\]
\end{theo}
\begin{proof} The first two as well as the last two inclusions follow
  from the general properties of the closure operators,  
the third is the content of Lemma \ref{lema:s1}, that is supported
mainly on the basic reduction rule of an ${\mathcal {AKS}}$
(Definition \ref{defi:aks}, \eqref{item:reduction}). The fourth
follows from Lemma \ref{lema:peprima}.\eqref{item:peprimamain} that is
based on the properties of the combinator $\cE$. See also
\cite[Section 2, \textsection 2.3]{kn:ocar}.
\end{proof}
\bigskip
\item \label{item:recupadj} Next we show that in the case that we work
  on ${\mathcal P}_\perp$, we recuperate partially the adjunction
  property --compare with Paragraph \ref{item:conductor}--. The name
  \emph{adjunctor} for the combinator $\cE$ is justified by its role
  in the inclusions below. See \cite[Section 2, \textsection 2.3,
    Theorem 2.13]{kn:ocar} for a proof that does not use the new
  products introduced in this work, but that is more precise in the
  sense that uses the combinator $\cE$ instead of $\cE\cE$ as it is
  used below.

\begin{theo}\label{theo:fulladj} Assume that $L \in {\mathcal
  P}_\perp(\Lambda), P,R \in {\mathcal P}_\perp(\Pi)$. Then 
\begin{enumerate}
\item If $L \leadsto_\perp R\quad \subseteq\quad P \quad\text{then}\quad
 R\quad \subseteq\quad P \ast_\perp L$.
\item If $R\quad \subseteq\quad P \ast_\perp L \quad\text{then}\quad L
\leadsto_\perp R\quad \subseteq\quad \overline{\eta P}\quad =\quad
\overline{\{\cE\cE\}^\perp 
  \ast_\bullet {}^\perp P}\quad  \subseteq\quad 
  \{\cE\cE\}^\perp\ast_\perp{}^\perp P$. 
 \end{enumerate}  
\end{theo}
\begin{proof} We concentrate in the second assertion: it follows from 
Theorem \ref{theo:clubsuitE}, that $R \subseteq P \ast_\perp L
\subseteq P \,\clubsuit\, L \subseteq \eta P \ast_\bullet L$. From the
adjunction property for $\ast_\bullet$ and $\leadsto_\bullet$ --see
Paragraph \ref{item:conductor}-- we deduce that $L \leadsto_\bullet R
\subseteq \eta P$. Taking double perpendicularity we deduce that $L
\leadsto_\perp R \subseteq \overline{\eta P}$ and the rest of the
inequalities follows from Observation \ref{obse:futureref}.
\end{proof}
\bigskip
\item\label{item:rewritingop}
\noindent The definitions of the basic operations given before: the
push--type operations in ${\mathcal P}(\Lambda) \times {\mathcal
  P}(\Pi) \rightarrow {\mathcal P}(\Pi)$ and the conductor--type
operations in ${\mathcal P}(\Pi) \times {\mathcal P}(\Lambda)
\rightarrow {\mathcal P}(\Pi)$, can be transferred to ${\mathcal
  P}(\Pi) \times {\mathcal P}(\Pi) \rightarrow {\mathcal P}(\Pi)$ by
taking perpendiculars in the first variable or in the second --in
other words composing with the map $(-)^\perp : {\mathcal P}(\Lambda)
\rightarrow {\mathcal P}(\Pi)$ in the first or in the second
variable--.  Next we look at the basic properties of this redefined
operations.
%Assume that $P,Q,R \in {\mathcal P}_\perp(\Pi)$. 
\begin{defi}\label{defi:newop} Given the push--type functions
  $\leadsto,\leadsto_\bullet,\leadsto_\perp: {\mathcal P}(\Lambda)
  \times {\mathcal P}(\Pi) \rightarrow {\mathcal P}(\Pi)$ we define
  the new operations: $\to,\to_\bullet,\to_\perp: {\mathcal P}(\Pi)
  \times {\mathcal P}(\Pi) \rightarrow {\mathcal P}(\Pi)$ as shown in
  the table below. 
\medskip
\medskip  
\begin{center}
\begin{tabular}{rclcrcl}
  $P$& $\rightarrow$&$
  Q$&$:=$&${}^\perp P$&$\leadsto$&$Q$\\ 
  $P$&$\rightarrow_\bullet$&$Q$&$:=$&${}^\perp P$&$\leadsto_\bullet$&$Q$\\ 
  $P$&$\rightarrow_\perp$&$Q$&$:=$&${}^\perp P$&$\leadsto_\perp$&$Q$  
\end{tabular}
\end{center} 
Given the conductor--type operations $\ast, \ast_\bullet, \ast_\perp$
and the $\clubsuit$--operation, all of type ${\mathcal P}(\Pi)
\times {\mathcal P}(\Lambda) \rightarrow {\mathcal P}(\Pi)$; we define
the new operations: $\circ, \circ_\bullet, \circ_\perp, \diamond:
{\mathcal P}(\Pi) \times {\mathcal P}(\Pi) \rightarrow {\mathcal
  P}(\Pi)$ as shown in the table below.
\medskip
\medskip
\begin{center}
\begin{tabular}{rclcrcl}
  $P$&$\circ$&$Q$&$:=$&$P$&$\ast$&
  ${}^\perp Q$\\
  $P$&$\circ_\bullet$&$Q$&$:=$&$P$& 
  $\ast_\bullet$&${}^\perp Q$\\ 
  $P$&$\circ_\perp$&$Q$&$:=$&$P$&$\ast_\perp$&${}^\perp Q$\\ 
  $P$&$\diamond$&
  $Q$&$:=$&$P$&$\clubsuit$&${}^\perp Q$
\end{tabular}
\end{center} 
\end{defi}
Thus $\eta P= \{\cE\cE\}^\perp \circ_\bullet P$.
We reformulate the properties of the new operations:
\begin{enumerate}
\item \label{item:monotony2} \emph{Monotony.} The operations
  $\circ,\,\circ_\bullet,\,\circ_\perp,\, \diamond$ are monotone in
  both variables and
  $\rightarrow,\,\rightarrow_\bullet,\,\rightarrow_\perp$ are
  antitone in the first variable and monotone in the second.
\item \label{item:adj} \emph{Adjunctions.}
\begin{eqnarray*} Q \rightarrow R \subseteq
  P\quad&\text{if and only if}\quad R \subseteq P \,\circ\, Q; \\ Q
  \rightarrow_\bullet R \subseteq P\quad&\text{\,if and only if}\quad
  R \subseteq P \,\circ_\bullet Q.
\end{eqnarray*} 
\item \label{item:moreadj}\emph{Two half adjunctions.}
\begin{eqnarray*}
\quad \text{If}\quad &Q \rightarrow_\perp R \subseteq P\quad
\text{then}&\quad R \subseteq P \,\circ_\perp Q;\\ \quad
\text{If}\quad &R \subseteq P \,\circ_\perp Q\quad \text{then}&\quad
Q\rightarrow_\perp R \subseteq \overline{\eta P}=
\overline{\{\cE\cE\}^\perp \,\circ_\bullet P} \subseteq
\{\cE\cE\}^\perp \,\circ_\perp P.
\end{eqnarray*} 

\item \label{item:ineq} \emph{Inclusion relations.}
\begin{equation*} 
  \begin{split}
  P \,\circ_\bullet Q\quad  
  \subseteq\quad P \,\circ\, Q\quad \subseteq\quad P\,\circ_\perp
  Q\quad \subseteq\quad P 
  \diamond Q\quad \subseteq \\
  \subseteq\quad \eta P \,\circ_\bullet Q\quad
  \subseteq\quad 
  \eta P \,\circ\, Q 
  \quad \subseteq\quad \eta P\,\circ_\perp Q ;
  \end{split}
\end{equation*}
\medskip 
\begin{equation*}P \rightarrow Q \subseteq P
  \rightarrow_\bullet Q \subseteq P \rightarrow_\perp Q.
\end{equation*}   
\item \label{item:comb2} \emph{Properties of the combinators}
\begin{eqnarray}\label{eqn:comb2}
\cK\perp&P\to(Q\to P)\\ \cS\perp&P\to(Q\to (R\to((P\circ
R)\circ(Q\circ R)))\\ \cE\perp&P\to(Q\to(P\diamond Q))
\end{eqnarray}

\end{enumerate} 

\begin{defi} \label{defi:finaloperations}Assume that ${\mathcal K}$ is
  an ${\mathcal {AKS}}$ 
\begin{enumerate}
\item The operations $\circ, \rightarrow: {\mathcal P}(\Pi) \times
  {\mathcal P}(\Pi) \rightarrow {\mathcal P}(\Pi)$ are called the
  \emph{application} and the \emph{implication} in~${\mathcal
    P}(\Pi)$.
\item The operations $\circ_\bullet, \rightarrow_\bullet: {\mathcal
  P}(\Pi) \times {\mathcal P}(\Pi) \rightarrow {\mathcal P}(\Pi)$ when
  restricted to ${\mathcal P}_\bullet(\Pi) \times {\mathcal
    P}_\bullet(\Pi)$ yield maps from ${\mathcal P}_\bullet(\Pi) \times
  {\mathcal P}_\bullet(\Pi) \rightarrow {\mathcal P}_\bullet(\Pi)$ and
  are called the \emph{application} and the \emph{implication} in
  ${\mathcal P}_\bullet(\Pi)$.
\item The operations $\circ_\perp, \rightarrow_\perp: {\mathcal
  P}(\Pi) \times {\mathcal P}(\Pi) \rightarrow {\mathcal P}(\Pi)$ when
  restricted to ${\mathcal P}_\perp(\Pi) \times {\mathcal
    P}_\perp(\Pi)$ yield maps from ${\mathcal P}_\perp(\Pi) \times
  {\mathcal P}_\perp(\Pi) \rightarrow {\mathcal P}_\perp(\Pi)$ and are
  called the \emph{application} and the \emph{implication} in
  ${\mathcal P}_\perp(\Pi)$.
\end{enumerate}
\end{defi}
\item\label{item:productsequal}
\noindent
In accordance with Definition \ref{defi:sharp} the implication of a
complete $\qoca$, induces an application like operation that we called
$\sharp$. Given ${\mathcal K} \in {\mathcal {AKS}}$, as defined in
\cite{kn:ocar}, there is a $\koca$ ${\mathcal A}_{{\mathcal K}}$ which
is associated to ${\mathcal K}$. The theorem that follows shows that
there is a close connection between the $\sharp$ defined in ${\mathcal
  A}_{\mathcal K}$ and the operation $\circ_\bullet$ of ${\mathcal
  K}$. The corresponding considerations for the case of the other
operations are not too interesting because in those cases, as there is
a full adjunction, the results of Theorem \ref{theo:miquel} guarantee
that the $\sharp$--product and the $\circ$--product, coincide.
\begin{theo} \label{theo:productsequal} For an $\qoca$ of the form
  ${\mathcal A}_{\mathcal K}$ where ${\mathcal K}$ is an ${\mathcal
    {AKS}}$ and for all $a,b \in A_\perp={\mathcal P}_\perp(\Pi)$ then
  $\overline{a\circ_\bullet b}= a\,\sharp\,b$. See Paragraph
  $\ref{item:conductor}$ for the definition of $a\circ_\bullet b$ and
  Definition \ref{defi:sharp} for the meaning of $a\sharp b$.
\end{theo}
\begin{proof} Consider: 
\[
\begin{array}{r}
 a \,\sharp\, b = \operatorname{inf}\{c \in {\mathcal P}_\perp(\Pi): a
 \leq b \rightarrow_\perp c\}=\sup_\perp\{c \in {\mathcal
     P}_\perp(\Pi): b \rightarrow_\perp c \subseteq
   a\}=\\
   \\
   \\
   \overline{\bigcup\{c \in {\mathcal P}_\perp(\Pi): b \rightarrow_\perp
   c \subseteq 
   a\}}=\overline{\bigcup\{c \in {\mathcal P}_\perp(\Pi): b
   \rightarrow c \subseteq a\}}=\\
 \\
 \\
\overline{\bigcup\{c \in {\mathcal
     P}_\perp(\Pi): c \subseteq a \circ b\}}= \overline{\bigcup\{c \in
   {\mathcal P}_\perp(\Pi): c \subseteq a \circ_\bullet b\}},
\end{array}
\] 
where the fourth equality is justified by the
considerations in paragraph \ref{item:completeness}, item (3), the
fifth is consequence of \ref{obse:table}, \ref{item:apraro} item (b),
the sixth is by 
\ref{item:hap} item (1) and the last one is consequence of
\ref{obse:table}, \ref{item:apraro}, item (1). 

\noindent 
Now, if $\pi \in a\circ_\bullet b$ then $\overline{\pi} \subseteq
a\circ_\bullet b$ and it follows by the above equality that
$\overline{\pi} \subseteq a\, \sharp\, b$, and hence $a\circ_\bullet b
\subseteq a\, \sharp\, b$ and also $\overline{a\circ_\bullet b}
\subseteq a\, \sharp\, b$.

\noindent
Conversely, it is clear that $\bigcup\{c \in {\mathcal P}_\perp(\Pi):
c \subseteq a \circ_\bullet b\} \subseteq a \circ_\bullet b$, and
taking closures we obtain that $a\, \sharp\, b \subseteq \overline{a
  \circ_\bullet b}$.
\end{proof}

\begin{obse}\  
\begin{enumerate}
\item Notice that in accordance with Paragraph \ref{item:nonclosed},
  the product $a \circ_\bullet b$ need not be closed with respect to
  the bar operator. Hence in general the operations $\circ_\bullet$
  and $\sharp$ are different. The above proof guarantees that in
  general for all $a,b \in {\mathcal P}_\perp(\Pi)$, we have that $a
  \circ_\bullet b \subseteq a\,\sharp\, b$.
\item In general for all $a,b \in {\mathcal P}_\perp(\Pi)$, we have
  that: \[a \circ_\bullet b \subseteq a\,\sharp\, b=
  \overline{a\circ_\bullet b} \subseteq a \circ b.\]
\end{enumerate}
\end{obse}
%\bigskip
%\section{From ${\mathcal {AKS}}$s to $\oca$s: three
%guises}\label{section:guises} \bigskip
\section{From ${\mathcal {AKS}}$s and $\ioca$s to Heyting preorders}
\label{section:akstohpo}
\item In \cite{kn:ocar} we interpreted the main results appearing in
  \cite{kn:streicher}, as a triangle of constructions between the main
  objects under consideration:

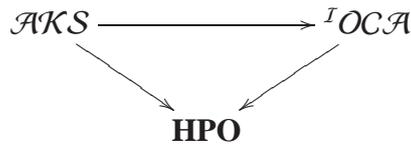
\[\xymatrix{\mathcal{AKS} \ar[rr] \ar[rd]&&\mathcal{{}^{\mathcal I}OCA}\ar[ld]\\
 &\bf {HPO}&}\] 
\noindent
 The left diagonal
arrow in the diagram is basically Streicher's construction in
\cite{kn:streicher}, i.e a map between abstract Krivine structures and
preorders, and the triangle can be interpreted as a factorization of this map.  

The horizontal arrow (compare with \cite{kn:ocar}) represents two
possible constructions of an $\oca$, one using the operations
$(\circ,\to)$ and the other $(\circ_\perp,\to_\perp)$. The first are
defined in ${\mathcal P}(\Pi)$ and the others in ${\mathcal
  P}_\perp(\Pi)$.

In this section we intend to perform a similar construction of an
horizontal map, using the operations $(\circ_\bullet,\to_\bullet)$ on
${\mathcal P}_\bullet(\Pi)$, and thereby obtaining a diagram similar to
the one above (with small changes due to the new operations):
\[\xymatrix{\mathcal{AKS} \ar[rr] \ar[rd]&&\mathcal{{}^{\mathcal F}OCA}\ar[ld]\\
 &{\bf {HPO}}&}\]

\item \label{item:constperp}
In this paragraph we start by recalling briefly the construction
  of the arrows of the first
  diagram, i.e. the construction of a $\ioca$ from an $\mathcal {AKS}$
  and of a Heyting preorder from a ${\ioca}$, see
  \cite[Section 5, Definition 5.10, Theorem 5.11; Section 4,
    Definition 4.13, Theorem 4.15]{kn:ocar} and the construction of a
  Heyting preorder from an $\mathcal {AKS}$.  
%% in order to adapt it to
%%   the structure with the new \emph{bullet} operation $\bullet$.
%% In \cite{kn:ocar} the horizontal arrow considered is
%% as follows.
\begin{enumerate}
\item \emph{From $\mathcal{AKS}$ to $\ioca$.} For ${\mathcal K}$ an
  ${\mathcal {AKS}}$, the following elements produce a $\ioca$ denoted
  as ${\mathcal A}_{{\mathcal K}\perp}$: $A_\perp={\mathcal
    P}_\perp(\Pi)$, $a \leq b \in A_\perp$ iff $a \supseteq b$; and $a
  \circ_\perp b,\, a\rightarrow_\perp b\in A_\perp$ are as in
  Paragraph \ref{item:rewritingop}, Definition
  \ref{defi:finaloperations}, $\ck=\{\cK\}^\perp,\,\cs=\{\cS\}^\perp,
  \ce=\{\cE\}^\perp$. Moreover, $\Phi=\{P \in \mathcal P_{\perp}(\Pi):
  \exists t \in \mathrm{QP}\,\, t \perp P\}$.
\item \emph{From $\ioca$ to {\bf HPO}.}
If ${\mathcal A}$ 
  is a $\ioca$ with with filter $\Phi$ and maximum element
  $\top\in\Phi$, define:
 \begin{enumerate}
\item The relation $\sqsubseteq$ in $A$: $a
  \sqsubseteq b, {\textrm{if and only if}}\,\, \exists f \in
  \Phi$ such that $f a \leq b$,
\item A map $\wedge: A \times A \rightarrow A$ as $a \wedge b:= {\cp} ab$.
\end{enumerate}
If ${{\mathcal A}}$ is a ${\ioca}$ then defining $H_{\mathcal A}:=(A,
\sqsubseteq, \wedge, \rightarrow)$, it is clear that it is a Heyting
preorder where $\sqsubseteq, \wedge$ are as above and $\rightarrow$ is
the original implication of ${\mathcal A}$.  For the definition of
$\cp$ see \cite[Section 3, Definition 3.5]{kn:ocar} where it is
defined as the combinator $\cp:=\lambda^*x\lambda^*y\lambda^*z(zxy)$
using the combinatory completeness to garantee that it belongs to
$\Phi$.  An element $f$ as above is said to be ``a realizer of the
relation $a \sqsubseteq b$'' and write this assertion as $f \Vdash a
\sqsubseteq b$.  See \cite[Section 4, Lemma 4.14]{kn:ocar} for the
proof that it is a meet semilattice and \cite[Section 4, Theorem
  4.15]{kn:ocar} for the rest of the proof.
\item \label{item:akshpo} \emph{From $\mathcal{AKS}$ to {\bf HPO}.}  Let ${{\mathcal
    K}}=(\Lambda,\Pi,\Perp, \operatorname{push}, \operatorname{app},
  \cK, \cS, \mathrm{QP})$ be an abstract Krivine structure. We define
  the relation $\sqsubseteq_\perp$ in ${\mathcal P}(\Pi)$ as follows:
  $\label{eq:entailment11} P,Q \in {{\mathcal P}}(\Pi)\,, \quad
  P\sqsubseteq_\perp Q \Leftrightarrow \exists t\in
  {\mathrm{QP}}\,\,\, t{\perp} P{\rightarrow_\perp} Q$.  An element $t
  \in \mathrm{QP}$ as above is said to be ``a realizer of the
  relation'' $P\sqsubseteq_\perp Q$''. It follows that for $P,Q \in
  {{\mathcal P}}(\Pi)$ we have that $t \in \Lambda$ is a realizer of
  the relation $P \sqsubseteq Q$ if and only if it is a realizer of $P
  \sqsubseteq_\perp Q$. Prima facie we have only considered a preorder
  $\sqsubseteq_\perp$ in ${\mathcal P}_{\perp}(\Pi)$ and thus produced an
  object in ${\bf Ord}$, called ${\mathcal H}_{{\mathcal K} \perp}$. In
  \cite[Theorem 5.11]{kn:ocar} it is proved that $\mathcal H_{\mathcal
    K \perp}$ and $\mathcal H_{{\mathcal A}_{\mathcal K \perp}}$ are
  isomorphic, and that implies that the first one is in fact an
  element in ${\bf HPO}$.
\end{enumerate}
\item Next we want to construct a triangle of the second kind
  mentioned above. Regarding the horizontal arrow, for ${\mathcal K}
  \in {\mathcal {AKS}}$, we construct an ordered combinatory algebra
  that we call ${\mathcal A}_{{\mathcal K}\bullet} \in \foca$. It has
  as basic set ${\mathcal{P}}_{\bullet}(\Pi)$, operations
  $\circ_\bullet\,,\,\to_\bullet$ and combinators $\ck_\bullet\,,\,
  \cs_\bullet$ where $\ck_\bullet=\{\cE\cK\}^\perp,\,\cs_\bullet=
  \{\cE((\cB\cE)\cS)\}^\perp$ (see Section \ref{section:opcomb}).
  This ${\foca}$ has a filter that is the natural extension of the
  filter of $\mathcal A_{\mathcal K \perp}$.

Compared with the above construction of ${\mathcal A}_{{\mathcal
    K}\perp}$, if we choose to work with ${\mathcal A}_{{\mathcal
    K}\bullet}$ we will enjoy the following benefits:
  \begin{enumerate}
    \item The construction of the associated tripos is more direct
      than in the case of ${\mathcal A}_{{\mathcal K}\perp}$. This is
      due to the fact that in this context we do not have to recurr to
      an $\eta$-expansor or adjunctor $\ce$ --compare with the
      construction appearing in \cite{kn:streicher} or
      \cite{kn:report}--.
    \item The use ${\mathcal{P}}_{\bullet}(\Pi)$, that is a priori
      ``larger'' than ${\mathcal{P}}_{\perp}(\Pi)$ as the set of
      falsity values, might have some conceptual advantages. Besides
      producing more programs expressible in the semantics, it is
      closer to Krivine's original definition that took
      ${\mathcal{P}}(\Pi)$ as the set of falsity values.
  \end{enumerate}

Then we revisit the construction of the left diagonal arrow in our context, adapting the case of ${\mathcal P}_{\perp}$ to the context of the new closure operator i.e. the case of ${\mathcal P}_{\bullet}$. 

Finally, with respect to the remaining arrow in the triangle, the right
diagonal arrow, there is not much to be done as the construction performed  for a $\foca$ rather than a $\ioca$ is identical.
\item \label{item:noadjunctor} Next we present the precise statements
  and the proofs of the above assertions.
\begin{theo} \label{theo:akstoocabullet} 
Let ${\mathcal{K}}= (\Lambda, \Pi, \bbot, \operatorname{push}, \cK,
\cS,\mathrm{QP}) \in {\mathcal AKS}$, $\ck_\bullet=\{\cE\cK\}^\perp$,
$\cs_\bullet= \{\cE((\cB\cE)\cS)\}^\perp \in
   {\mathcal{P}}_{\bullet}(\Pi)$ and $\Phi=\{P \in {\mathcal
     P}_{\bullet}(\Pi)|\,\exists\, t \in \mathrm{QP}\,,\, t \perp P \}
   \subseteq {\mathcal P}_{\bullet}(\Pi)$. Then ${\mathcal
     A}_{{\mathcal K}\bullet} =({\mathcal P}_{\bullet}(\Pi),
   \circ_\bullet, \rightarrow_\bullet, \supseteq, \ck_\bullet,
   \cs_\bullet,\Phi)$ is a $\foca$.
\end{theo}
\begin{proof} Clearly $\ck_\bullet, \cs_\bullet \in{\mathcal{P}}_\perp(\Pi)\subseteq
  {\mathcal{P}}_\bullet(\Pi)$ because both are perpendiculars of
  subsets of $\Lambda$. The operation $\circ_\bullet$ is monotonic in
  both arguments and $\to_\bullet$ is antimonotonic on the first
  argument and monotonic in the second --see Paragraph
  \ref{item:rewritingop}--.  To prove that $\ck_\bullet$ and
  $\cs_\bullet$ satisfy the requiered properties we proceed as
  follows.
\begin{enumerate}
\item The inequality $\ck_\bullet \circ_\bullet P\circ_\bullet
  Q\supseteq P$ holds for generic subsets $P,Q \in {\mathcal
    P}_\bullet(\Pi)$.  Indeed, if we take $p \in {}^\perp P,\, \pi \in
  P$ implies: see Definition
  \ref{defi:aks}.\eqref{item:reduction}.(b), that $\cK \perp p\cdot q
  \cdot \pi$ for all $q \in \Lambda$ in particular for all $q \perp Q$
  and consequently we have that in the same sets: $\cK p \perp q \cdot
  \pi$.  In other words we have that $\cK p \in {}^\perp ({}^\perp Q
  \cdot P)= {}^\perp(Q \rightarrow P)={}^\perp \bigl((Q \rightarrow
  P)^{\widehat{}}\,\,\bigr)={}^\perp(Q \rightarrow_\bullet P)$, the
  equalities being guaranteed by the considerations in Paragraph
  \ref{item:polarclosure},\eqref{item:hatperp=perp},(2) and by Definition
  \ref{defi:newop}. Hence, it follows from Paragraph
  \ref{item:morecombinators},\eqref{item:newcombprop} that for all $p
  \perp P,\,\, \rho \in Q \rightarrow_\bullet P$, $\cE \perp \cK\cdot
  p\cdot \rho$ and also that $\cE\cK \perp p\cdot \rho$. In the same
  manner than before, this implies that $\cE\cK \in {}^\perp(P
  \rightarrow_\bullet (Q \rightarrow_\bullet P))$.  This inequality is
  equivalent with $\ck_\bullet \supseteq P \rightarrow_\bullet (Q
  \rightarrow_\bullet P)$ and the adjunction property for
  $(\rightarrow_\bullet, \circ_\bullet)$ implies that $\ck_\bullet
  \circ_\bullet P\circ_\bullet Q\supseteq P$.
\item Now we have to prove that $\cs_\bullet \circ_\bullet
  P\circ_\bullet Q\circ_\bullet R\supseteq (P\circ_\bullet
  R)\circ_\bullet (Q\circ_\bullet R)$ for generic $P,Q,R \in {\mathcal
    P}_\bullet(\Pi)$.

It follows from the adjunction property that the above is equivalent
with $\cs_\bullet \supseteq P\to_\bullet (Q\to_\bullet (R\to_\bullet
(P\circ_\bullet R)\circ_\bullet (Q\circ_\bullet R)))$.  Take $p \in
{}^\perp P, q \in {}^\perp Q, r \in {}^\perp R$. We use the inequality
: $X \circ_\bullet Y \subseteq X \diamond Y$ to guarantee in succesive
steps that $pr \perp P \circ_\bullet R, qr \perp Q \circ_\bullet R$
and also that $(pr)(qr) \perp (P \circ_\bullet R)\circ_\bullet (Q
\circ_\bullet R)$.  It follows from Definition \ref{defi:aks},
\eqref{item:reduction} that $\cS \perp p\cdot q \cdot r \cdot \pi$ for
all $p,q,r$ as above and for all $\pi \in (P\circ_\bullet R)
\circ_\bullet (Q\circ_\bullet R)$. It follows then that $\cS pq \perp
r\cdot \pi$ and then in the same manner than before that $\cS p q
\perp R\to_\bullet (P\circ_\bullet R)\circ_\bullet (Q\circ_\bullet
R)$. Let $\rho \in R\to_\bullet (P\circ_\bullet R)\circ_\bullet
(Q\circ_\bullet R)$ using the properties for the combinators $\cE,
\cB$ appearing in Paragraph
\ref{item:morecombinators},\eqref{item:newcombprop} we deduce first
that $\cE \perp (\cS p)\cdot q \cdot \rho$ for all $p,q,\rho$ as
above, and then that $\cB \perp \cE\cdot \cS \cdot p \cdot q \cdot
\rho$.  Then, $(\cB \cE)\cS p \perp q\cdot \rho$ and as before write
down this property as: $(\cB \cE)\cS p \perp Q\to_\bullet
(R\to_\bullet (P\circ_\bullet R)\circ_\bullet (Q\circ_\bullet R))$. If
call $\gamma$ a generic element in $Q\to_\bullet (R\to_\bullet
(P\circ_\bullet R)\circ_\bullet (Q\circ_\bullet R))$, from $(\cB
\cE)\cS p \perp \gamma$ we deduce first that $\cE \perp ((\cB
\cE)\cS)\cdot p \cdot \gamma$ and then that $\cE ((\cB \cE)\cS)\perp p
\cdot \gamma$.  As before, a perpendicularity relation like the above
means that $\cE ((\cB \cE)\cS) \perp P\to_\bullet (Q\to_\bullet
(R\to_\bullet (P\circ_\bullet R)\circ_\bullet (Q\circ_\bullet R)))$
and then that $\{\cE ((\cB \cE)\cS)\} \subseteq {}^\perp(P\to_\bullet
(Q\to_\bullet (R\to_\bullet (P\circ_\bullet R)\circ_\bullet
(Q\circ_\bullet R))))$.  Then: $\{\cE ((\cB \cE)\cS)\}^\perp \supseteq
({}^\perp(P\to_\bullet (Q\to_\bullet (R\to_\bullet (P\circ_\bullet
R)\circ_\bullet (Q\circ_\bullet R))))^\perp \supseteq P\to_\bullet
(Q\to_\bullet (R\to_\bullet (P\circ_\bullet R)\circ_\bullet
(Q\circ_\bullet R)))$.
\item To prove that $\Phi \subseteq A$ is a filter that contains
  $\ck_{\bullet},\cs_{\bullet}$ we proceed as follows. The subset
  $\Phi$ is closed under application. Indeed, if $f,g \in \Phi$, we
  have that $t_f\in {}^\perp f \cap \mathrm{QP}$ and $t_g\in {}^\perp
  g \cap \mathrm{QP}$, then $t_ft_g\in {}^\perp f {}^\perp g \cap
  {\mathrm{QP}}$.  It follows directly from Lemma \ref{lema:s1},
  \eqref{item:s14} that ${}^\perp f {}^\perp g \subseteq {}^\perp(f
  \circ_{\bullet} g)$, and hence $\Phi$ is closed under application.
  Moreover, $\ck_{\bullet}$ is the set of elements perpendicular to a
  fixed element of $\mathrm{QP}$, and the same for $\cs_{\bullet}$. It
  is clear --by the definition of $\Phi$-- that set of this kind are
  elements of $\Phi$.
\end{enumerate}
\end{proof}
\item\label{item:streicherbullet} In this paragraph, given $\mathcal K$ an ${\mathcal{AKS}}$ and
  taking into account the realizability relation, we define a preorder
  $\mathcal H_{{\mathcal K}\bullet}$ based on the set ${\mathcal
    P}_\bullet(\Pi)$, that is similar to the one defined by Streicher
  in \cite{kn:streicher} and summarized in Paragraph
  \ref{item:constperp},\eqref{item:akshpo} and used in \cite{kn:ocar}
  to define realizability in terms of $\ioca$s.

\begin{defi}\label{defi:squareorder} 
Let ${{\mathcal K}}=(\Lambda,\Pi,\Perp, \operatorname{push}, \operatorname{app}, \cK,
  \cS, \mathrm{QP})$ be an abstract Krivine
structure. We define the relation $\sqsubseteq_\bullet$ in ${\mathcal
  P}(\Pi)$ as follows:
\begin{equation*}\label{eq:entailment11} P,Q \in {{\mathcal P}}(\Pi)\,, \quad P\sqsubseteq_\bullet Q \Leftrightarrow \exists t\in
  {\mathrm{QP}}\,\,\, t{\perp} P{\rightarrow_\bullet} Q=({}^\perp P).Q,
\end{equation*} for $P, Q \in {{\mathcal P}}_\bullet(\Pi)$. An element $t
\in \mathrm{QP}$ as above is said to be ``a realizer of the
relation'' $P\sqsubseteq_\bullet Q$''.
\end{defi}
\begin{obse}\label{remark:equal}  From Definition \ref{defi:erre} and 
Theorem \ref{item:hatperp=perp} it follows that for $P,Q \in
{{\mathcal P}}(\Pi)$ we have that ${}^\perp (P \rightarrow Q)={}^\perp
(P \rightarrow_\bullet Q)={}^\perp (P \rightarrow_\perp Q)$. Hence, $t
\in \Lambda$ is a realizer of the relation $P \sqsubseteq Q$ if and
only if it is a realizer of $P \sqsubseteq_\bullet Q$ if and only if
it is a realizer of the relation $P \sqsubseteq_\perp Q$.
\end{obse}
\begin{lema}\label{lem:aks-ord} Let ${\mathcal K}$ be an abstract
  Krivine structure, then the relation $\sqsubseteq$ is a preorder on
  ${{\mathcal P}}(\Pi)$ and the relation $\sqsubseteq_\bullet$ is a
  preorder on ${\mathcal P}_\bullet(\Pi)$
\end{lema}
\begin{proof} The first assertion appears proved in \cite[Section 4, Lemma 4.12]{kn:ocar} and the second follows from the above Observation \ref{remark:equal}.
\end{proof}

\begin{lema}\label{lem:eqpo} The canonical inclusion
$({\mathcal P}_{\bullet}(\Pi),\sqsubseteq_\bullet) \hookrightarrow
  ({\mathcal P}(\Pi),\sqsubseteq)$ is an equivalence of preorders.
\end{lema}
\begin{proof} By the comments at the beginning of 
Paragraph ~\ref{item:recall}, it suffices to show that the inclusion
is order reflecting and essentially surjective. Since the order on
${\mathcal P}_{\bullet}(\Pi)$ is defined as restriction of the order
on ${\mathcal P}(\Pi)$, the first assertion is clear.

As $({}^\perp P)^\perp \in {\mathcal P}_\bullet(\Pi)$ for all $P \in
{\mathcal P}_\bullet(\Pi)$, once we prove that
$P\sqsubseteq_{\bullet}({}^\perp P)^\perp$ and $({}^\perp
P)^\perp\sqsubseteq_{\bullet} P$, the fact that the inclusion is
essentially surjective follows directly.

To prove the above inclusions, first recall from \cite[Section 2,
  Lemma 2.12, (4)]{kn:ocar} that for all $Q \subseteq \Pi$ $\cI \perp
{}^\perp Q \cdot Q$, and then apply this perpendicularity relation to the situation where $Q:= ({}^\perp P)^\perp$ and $Q:=P$ respectively.
\end{proof}
\begin{defi} For $\mathcal K$ an $\mathcal{AKS}$ we define the preorder ${\mathcal H}_{\mathcal K \bullet}:=({\mathcal P}_{\bullet}(\Pi),
\sqsubseteq_\bullet) \in {\bf Ord}$.
\end{defi}
\begin{obse} We may state the above result, in terms of Streicher's construction of the map  $\mathcal {AKS} \rightarrow \ioca \rightarrow {\bf Ord}$, as guaranteeing that the three constructions: (1) $\mathcal K \rightarrow \mathcal A_{\mathcal K}$, (2) $\mathcal K \rightarrow \mathcal A_{\mathcal K\bullet}$, (3) $\mathcal K \rightarrow \mathcal A_{\mathcal K\perp}$;  composed with the construction from $\ioca \rightarrow {\bf HPO}$ produce equivalent preorders. 
\end{obse}
\item In this paragraph, we show that if ${\mathcal K}$ is in
  ${\mathcal {AKS}}$ the preorders ${\mathcal H}_{\mathcal K \bullet}$
  and ${\mathcal H}_{\mathcal A_{\mathcal K\bullet}}$ are
  isomorphic (see Paragraphs \ref{item:streicherbullet},\ref{item:noadjunctor} and the definition below). 

%% after recalling the construction $ \mapsto {\mathcal
%%     A}_{{\mathcal K}\bullet}$ appearing in \ref{item:noadjunctor} that
%%   produces a $\ioca$ from an we show that the indexed preorder
%%   associated to the ${\mathcal {AKS}}$ and the tripos associated to
%%   the $\ioca$ are isomorphic.

\begin{defi}If ${\mathcal A}$ is a $\foca$ with with filter $\Phi$ and
  maximum element $\top\in\Phi$, define: the relation $\sqsubseteq$ in
  $A$: $a \sqsubseteq b, {\textrm{if and only if}}\,\, \exists f \in
  \Phi$ such that $f a \leq b$ and a map $\wedge: A \times A
  \rightarrow A$ as $a \wedge b:= {\cp} ab$. We call ${\mathcal
    H}_{\mathcal A}=(A, \sqsubseteq, \wedge,\rightarrow)$ the
  Heyting preorder associated to ${\mathcal A}$.
\end{defi}
It is clear (compare with the definitions in Paragraph
\ref{item:constperp}), that if ${{\mathcal A}}$ is a ${\foca}$ one has
that ${\mathcal H}_{\mathcal A}$ is a Heyting preorder.

\begin{theo}\label{theo:pakstooca} Let ${\mathcal K}$ be an $\mathcal{AKS}$, then  ${\mathcal
    H}_{\mathcal K\bullet}$ and ${\mathcal H}_{\mathcal A_{\mathcal K\bullet}}$,
  are isomorphic preorders. In particular, we conclude that
  ${\mathcal H}_{\mathcal K\bullet}$ is a Heyting preorder.
\end{theo}
\begin{proof} 
In both cases the basic sets of the preorders is ${\mathcal
  P}_{\bullet}(\Pi)$.  We only have to 
check that the two definitions of the preorder coincide. The order
relation in in $\mathcal{H}_{\mathcal{A}_{\mathcal{K}\bullet}}$ is
  given by: 
$\exists P\in \Phi\,,\,P \circ_\bullet Q\supseteq R$, which
using the full adjunction can be formulated
equivalently as: $\exists P\in \Phi\;\,P\supseteq Q  \to_{\bullet} R$.

As to the equivalence we have that:
\begin{eqnarray*} & \exists P\in \Phi\,&,\,\quad
P\supseteq Q\to_{\bullet}R\\ \Leftrightarrow\quad & \exists
t\in {\mathrm{QP}}\,&,\,\quad \{t\}^\perp\supseteq
({}^\perp Q.R)^{\widehat{}}\\ \Leftrightarrow\quad &
\exists t\in {\mathrm{QP}}\,&,\,\quad t\perp
({}^\perp Q.R)^{\widehat{}} \\ \Leftrightarrow\quad &
\exists t\in {\mathrm{QP}}\,&,\,\quad
t\perp {}^\perp Q.R\\ \Leftrightarrow\quad & \exists
t\in {\mathrm{QP}}\,&,\,\quad t\perp{} Q\rightarrow R,\\
\end{eqnarray*} and the last line is the definition of the preorder in
${\mathcal H}_{\mathcal K\bullet}$.
\end{proof}

\item\label{item:inverse} Next we consider and adapt to the current context, the
  construction of an inverse of the horizontal map that appears in
  \cite[Definition 5.12]{kn:ocar}.

In this weaker context the map going in the oposite direction than
$\mathcal K \rightarrow {\mathcal A}_{\mathcal K \bullet}$, that
we denote as $\mathcal A \rightarrow {\mathcal K}_{\mathcal A
  \bullet}$, will not be a full inverse but only a \emph{Galois injection} (see Observation \ref{obse:galoisinj}):
\[\xymatrix{\mathcal{AKS} \ar@/_.5pc/[rr]_{\mathcal
  K \rightarrow {\mathcal A}_{\mathcal K \bullet}}
  \ar[rd]&&\mathcal{{}^{\mathcal
      F}OCA}\ar[ld]\ar@/_.5pc/[ll]_{ {\mathcal
      K}_{\mathcal A \bullet}\leftarrow {\mathcal A}}\\ &{\bf HPO}&}\]
\begin{defi}\label{def:aks-from-rdlp} Given a $\foca$ called 
${\mathcal A}=(A,\leq,\rapp,\rimp,{\ck},{\cs},\Phi)$,
we define the structure: ${\mathcal
  K}_{\mathcal A \bullet}=(\Lambda,\Pi,\Perp,\mathrm{app}, \mathrm{push}, \cK,\cS,\mathrm{QP})$
as follows. 
\begin{enumerate}
\item $\Lambda=\Pi:=A$;
\item $\Perp:=\,\leq$\,,\, i.e.\,\, $s\perp\pi :\Leftrightarrow
  s\leq\pi$;
\item ${\rapp}(s,t):=st$ \quad,\quad ${\mathrm{push}}(s,\pi):={\rimp}(s,\pi)=
  s\to\pi$;
\item $\cK := {\ck}\quad ,\quad \cS :={\cs}$;
\item $\mathrm{QP}:=\Phi$.
\end{enumerate}
\end{defi} 
\begin{theo} 
In the notations of Definition \ref{def:aks-from-rdlp},
  ${\mathcal K}_{\mathcal A \bullet}$ is an $\mathcal {AKS}$.
\end{theo}
\begin{proof} 
It is clear that $\mathrm{QP}$ is closed under application
  and contains $\cK,\cS$. Next we check the axioms
  concerning the orthogonality relation (see Definition
  \ref{defi:aks}). Substituting the above definitions, these axioms
  become:

\noindent $
\begin{array}{l@{\hspace{.5cm}}r@{\;\leq\;}l@{\quad\Rightarrow\quad}r@{\;\leq\;}l} \textrm{(S1)} &t & u\to\pi& tu&\pi \\
\textrm{(S2)} &t&\pi &{\ck}& t\to u\to\pi
\\ \textrm{(S3)} & tv(uv)&\pi &
  {\cs}
   & t\to u\to
   v\to\pi
\end{array}$

\medskip
(S1)\ follows from Definition \ref{defi:ioca},(PA).

(S2)\ is shown by the following derivation based on the definition of
$\ck$ (\ref{defi:ioca},(PK)) and the full adjunction,
i.e. \ref{defi:ioca},(PE)':
\[ t \leq  \pi \Rightarrow {\ck} tu\leq \pi \Rightarrow 
{\ck} t \leq u\to\pi \Rightarrow {\ck} \leq t\to u\to\pi.\]

(S3)\ is proved using repeatedly the full adjunction property as before as well as \ref{defi:ioca},(PS):
\[ tv(uv)\leq \pi \Rightarrow {\cs}
tuv\leq\pi \Rightarrow {\cs} tu \leq v\to\pi \Rightarrow {\cs}t
 \leq u\to v\to\pi \Rightarrow {\cs}
\leq t\to u\to v \to \pi.\]

\end{proof}
\begin{obse} To avoid confusion, when we view $A$ as the set $\Pi$ of the corresponding ${\mathcal{AKS}}$ $\mathcal K_{\mathcal A\bullet}$ 
we write it as $A_{\Pi}$ and when we view it as $\Lambda$ we write
$A_{\Lambda}$.
\end{obse}
\item In what follows we compare both constructions, considering the
  relationship between ${\mathcal A}$, the original $\foca$,  and the
  iterated construction of $\mathcal A_{\mathcal K_{\mathcal
      A\bullet}\bullet}$.
\begin{lema}\label{lem:galois}
Let $\mathcal A$ be a $\foca$   and ${\mathcal K}_{\mathcal A\bullet}$
the $\mathcal{AKS}$ of Definition~\ref{def:aks-from-rdlp}.
\begin{enumerate}
\item \label{lem:galois-u} For $C\subseteq A_{\Pi}$ we have ${}^\perp
  C ={\downarrow}(\inf C)$, and for $C\subseteq A_{\Lambda}$ we have
  $C^\perp={\uparrow}(\sup C)$.
\item \label{lem:galois-part} In particular{\emph {:}}
\begin{equation}\label{eq:galois-part}
\begin{array}{cccccccc}
  {}^{\perp}({\uparrow}a)&={\downarrow}a;
  {}^{\perp}({\downarrow}a)&=\perp; {}^\perp a&={\downarrow}a;
  ({}^{\perp}a)^{\perp}&= {\uparrow}a; \overline{{\uparrow}a}&=
  {\uparrow}a; \widehat{({\uparrow}a)}&={\uparrow}a &\mbox{for}\, a \in
  A_{\Pi}; \\ ({\downarrow}a)^\perp
  &={\uparrow}a;({\uparrow}a)^{\perp}&=\top; a^{\perp}&={\uparrow}a ;
  {}^{\perp}(a^{\perp})&={\downarrow}a; \overline{{\downarrow}a}&=
  {\downarrow}a; \widehat{({\downarrow}a)}&= {\downarrow}a &\mbox{for}\,
  a \in A_{\Lambda}.
\end{array}
\end{equation}

\item\label{lem:galois-pbot} If and $D\subseteq A_{\Pi}$, then
  $\overline{D}= {\uparrow}(\inf D)$ and
  $\widehat{D}=\bigcup\{{\uparrow}c: c \in D\}$. In other words:
  $\mathcal P_{\perp}(A_{\Pi})$ consists of the set of all principal
  filters of $A$ and $\mathcal P_{\bullet}(A_{\Pi})$ is the set of all
  subsets that are union of principal filters of $A$. In particular
  $\inf D =\inf(\widehat{D})=\inf(\overline{D})$.
%% $P_D=\bigcup\{{\uparrow} d: d \in D\}$, then
  %% ${\mathcal P}_\bullet(A)=\{P_D: D \subseteq A\}$.  consists
  %% precisely of
  
\item\label{lem:galois-infimp} For $C,D\in {\mathcal P}(A_{\Pi})$ we
  have $(\inf C\to \inf D) \leq \inf(C \rightarrow_{\perp} D)=\inf(C
  \rightarrow_{\bullet} D)=\inf(C\rightarrow D)$.
\end{enumerate}
\end{lema}
\begin{proof}
The verification of the first three items is clear, the remaining can be proved as follows.

$C\rightarrow D={}^\perp C . D \subseteq ({}^\perp C
. D)^{\widehat{}}=C\rightarrow_{\bullet} D \subseteq ({}^\perp C
. D)^{-}=C\rightarrow_{\perp} D$, then $\inf(C \rightarrow_{\perp}
D)=\inf(C \rightarrow_{\bullet} D)=\inf(C\rightarrow D)$ --see part
\eqref{lem:galois-pbot} of this Lemma. The inequality that remains to be  proved, is a consequence of the antimonotony of the arrow with respect to the first variable:  $\inf\{\inf C \rightarrow d: d \in D\} \leq \inf\{c \rightarrow d: c \leq \inf C\,,\,d \in D\} \leq \inf (C \rightarrow D)$. Clearly $\inf\{\inf C \rightarrow d: d \in D\}= \inf C \rightarrow \inf D$ by the preservation of limits by right adjoints and hence the conclusion follows. 
 
\end{proof}

\begin{theo}\label{theo:galois-inv} Let $\mathcal A$ be a $\foca$, ${\mathcal K}_{\mathcal A\bullet}$ 
the $\mathcal{AKS}$ of Definition~\ref{def:aks-from-rdlp} and
$\mathcal A_{\mathcal K_{\mathcal A\bullet}\bullet}= ({\mathcal
  P}_{\bullet}(A_{\Pi}), \circ_\bullet, \rightarrow_\bullet,
\supseteq, \ck_\bullet, \cs_\bullet,\Phi)$ the associated $\foca$
--Theorem \ref{theo:akstoocabullet}--.  Then, the functors:
\[\xymatrix{(A,\leq)\ar@<.5ex>[r]^-{\iota}& ({\mathcal P}_\bullet(A_{\Pi}),\supseteq )\ar@<.5ex>[l]^-{\rho}},\] 
$\iota:A\to{\mathcal P}_{\bullet}(A_{\Pi}),\, a\mapsto {\uparrow} a$
and $\rho: {\mathcal P}_{\bullet}(A_{\Pi})\to A,\, C\mapsto \inf C$,
form and adjoint pair (a Galois connection), i.e. if $a \in A$, $C
\subseteq A_{\Pi}\,,\, \widehat{C}=C$ then:
\begin{equation}\label{eqn:galois-con}a \leq \rho(C) \Leftrightarrow \iota(a) \leq C :\Leftrightarrow \iota(a) \supseteq C.
\end{equation}
and the unit of the adjunction $\iota \dashv \rho$ is an isomorphism
: $\operatorname{id}_A \cong \rho\,\iota$, while the counit is the
natural inclusion $\iota\,\rho \leq \operatorname{id}_{\mathcal
  P_\bullet(A)}$.
\end{theo}
\begin{proof} The codomain of $\iota$ lies in $\mathcal P_{\bullet}(A_{\Pi})$, see Lemma \ref{lem:galois}, (2)
where we proved that ${\uparrow}a \in \mathcal P_{\bullet}(A_{\Pi})$,
and it is clear that the maps $\iota$ and $\rho$ are monotonic.  As
$\inf({\uparrow} a)=a$ it follows that $\rho \iota =
\operatorname{id}_A$. Also for $C \subseteq A_{\Pi}\,,\,
\widehat{C}=C$, $\iota(\rho(C))= {\uparrow}(\inf C)=\overline{C}
\supseteq C= \operatorname{id}_{\mathcal P_{\bullet}(A_{\Pi})}(C)$.

Once we have the unit and counit, the proof of the assertion
\eqref{eqn:galois-con} follows directly.
\end{proof}

\begin{obse}\label{obse:galoisinj} \begin{enumerate}
\item A pair of preorders as above, equipped with a Galois connection \[\xymatrix{(A,\leq)\ar@<.5ex>[r]^-{\iota}& (B,\leq)\ar@<.5ex>[l]^-{\rho}}, \quad \mathrm{i.e.}\,\, \iota \dashv \rho,\]
with the property that the unit is an isomorphism, is sometimes called
a \emph{Galois injection}. It is easy to prove that to give such a
connection is equivalent to give an operator on $B$ satisfying the following axioms:
\begin{enumerate}
\item $b \leq b' \Rightarrow \zeta(b)\leq \zeta(b')$,
\item $\zeta(b)\leq b$,
\item $\zeta \zeta=\zeta$.
\end{enumerate}
To prove the above assertion just take $\zeta:=\iota\,\rho :B
\rightarrow B$. Conversely, given $\zeta: B \rightarrow B$ as above,
take $A=\zeta(B)$, $\rho=\zeta$  and $\iota=\operatorname{inc}$. 
\item
In our
case the operator $\zeta:{\mathcal P}_{\bullet}(A_{\Pi}) \rightarrow
{\mathcal P}_{\bullet}(A_{\Pi})$ is $\zeta(C)=\overline{C}$.
\end{enumerate}
\end{obse}
\item{\em From $\koca$\hspace*{1pt}s to $\mathcal{AKS}$s.}

Let us look again at the above triangle of constructions. 
 \[\xymatrix{\mathcal{AKS} \ar@/_.5pc/[rr]_{\mathcal
  K \rightarrow {\mathcal A}_{\mathcal K \bullet}}
  \ar[rd]_-{\mathcal S}&&\mathcal{{}^{\mathcal
      F}OCA}\ar[ld]^-{\mathcal H}\ar@/_.5pc/[ll]_{ {\mathcal
      K}_{\mathcal A \bullet}\leftarrow {\mathcal A}}\\ &{\bf HPO}&}\]

We want to show that the family of ${\bf HPO}$s produced by $\foca$s,
coincides with the family of ${\bf HPO}$s produced by Streicher's construction --named $\mathcal S$ in the diagram. This is a direct consequence of the following theorem:
\begin{theo}\label{theo:main3} In the notations above, the diagram \[\xymatrix{\mathcal{AKS} \ar[rd]_-{\mathcal S}&&\mathcal{{}^{\mathcal
      F}OCA}\ar[ld]^-{\mathcal H}\ar[ll]_{ {\mathcal
        K}_{\mathcal A \bullet}\leftarrow {\mathcal A}}\\ &{\bf
      HPO}&}\] is commutative, up to equivalence of preorders.  

%% The associated indexed triposes ${{\bf
%%       P}}({\mathcal A})$ and ${{\bf P}}_{\bullet}({\mathcal
%%     K}_{\mathcal A})$ are equivalent (see Paragraphs
%%   \ref{item:entailmentoca},\ref{item:entailmentaks}).
%% Definitions
%%   ~\ref{lem:oca-islat}-\ref{lem:oca-islat-func} and
%%   in~\ref{lem:aks-iord}-\ref{lem:aks-iord-func} respectively).
\end{theo}
\begin{proof}
If ${\mathcal K}$ is an ${\mathcal {AKS}}$ the map ${\mathcal S}$ is
defined as ${\mathcal S}({\mathcal K})= {\mathcal H}_{{\mathcal
    A}_{{\mathcal K}\bullet}}$.

Hence, we need to prove that the Heyting preorders $(A,\sqsubseteq)$
and $({\mathcal P}_{\bullet}(A_{\Pi}),\sqsubseteq_{\bullet})$, are
equivalent.  
%% Let $I$ be a set. The elements of ${{\bf P}}({\mathcal
%%   A})(I)$ are functions $\varphi:I\to A$, and the elements of ${\bf
%%   P}_{\bullet}({\mathcal K}_{\mathcal A})(I)$ are functions
%% ${\psi}:I\to{\mathcal P_{\bullet}}(A_\Pi)$.
Recall that the order relations are the following:
\begin{equation*}\label{eq:entailment1} C,C' \in
  {\mathcal P}_{\bullet}(A_\Pi)\,,\quad C \sqsubseteq_{\bullet} C' \Leftrightarrow
  \exists t\in {\mathrm{QP}}\;\,,\, t\perp
  C \rightarrow_{\bullet}C', 
\end{equation*}
\begin{equation*}\label{eq:entailment2} a,a' \in
  A\,,\quad a \sqsubseteq a' \Leftrightarrow
  \exists r\in \Phi\;\,,\, ra\leq a'. 
\end{equation*}
Consider the map $\rho:{\mathcal P}_{\bullet}(A_{\Pi}) \rightarrow A$, 
$\rho(C)=\inf C$.  

%% Define the natural transformation: \[\tau:{{\bf
%%     P}}_{\bullet}({\mathcal K}_{\mathcal A}) \Rightarrow {{\bf
%%     P}}({\mathcal A})\,,\, (I \stackrel{\psi}{\rightarrow}{\mathcal
%%   P}_{\bullet}(A_{\Pi})) \leadsto (I
%% \stackrel{\psi}{\rightarrow}{\mathcal
%%   P}_{\bullet}(A_{\Pi})\stackrel{\rho}{\rightarrow}A)=I \stackrel{\rho
%%   \psi}{\longrightarrow}A.\] Observe that the naturality of $\tau$
%% follows directly from the definition.
In order to guarantee that $\rho$ is an equivalence, we prove that
$\rho$ is order reflecting and essentially surjective (see Observation
\ref{obse:essentiallysur}).

The surjectivity condition for $\rho$ comes from the fact that the
adjunction $\iota \dashv \rho$ is a Galois injection (see Theorem
\ref{theo:galois-inv}).
%% Given $\varphi:I \rightarrow A$ consider
%% $\psi:=\iota \varphi: I \rightarrow {{\bf P}}_{\bullet}({\mathcal
%%   K}_{\mathcal A})$ where $\iota$ is as in Lemma \ref{lem:galois}. In
%% that situation $\rho \psi = \rho \iota \varphi= \varphi$ by the
%% mentioned lemma.
The fact that $\rho$ is order reflecting is guaranteed by the
following reasoning.  

The assertion $C
\sqsubseteq_{\bullet} C'$ in our situation means that: $\exists t \in
\mathrm{QP}: t \leq b\,\, \forall b \in \{s \rightarrow \pi, s\leq C,
\pi \in C'\}$.
Using the full adjunction we deduce that the condition $C
\sqsubseteq_{\bullet} C'$ is:
\begin{equation}\label{eqn:second2} \exists t \in
\mathrm{QP}=\Phi: ts \leq
           \pi\,\,\forall s \leq C\,,\, \pi \in C'.
\end{equation}

The assertion 
\begin{equation}\label{eqn:third}
\rho(C) \sqsubseteq \rho(C'),
\end{equation} 
can be written as: $\exists r \in \Phi : r\inf C \leq \inf C'$. It is
clear that if $s \leq C$, then $s \leq \inf C$, hence for such an $s
\leq C$, $rs \leq r \inf C \leq \inf C'$. We can take then $t:=r$ to
obtain a proof that from the condition \eqref{eqn:third}, we deduce
\eqref{eqn:second2}, i.e. that $\rho$ is order reflecting.
\end{proof}
\section{Constructing triposes from ordered structures}
\label{section:constripos}
\item\label{item:entailmentaks} In \cite[Section 5]{kn:ocar} we
  illustrated how to construct triposes from ordered combinatory
  algebras and from abstract Krivine structures and the relations
  between them. As the situation is almost the same, we very briefly revist that construction in our situation where we
  are dealing with $\foca$s.

Given an ${\mathcal {AKS}}$ called ${\mathcal K}$, define:
\begin{enumerate}
\item The \emph{entailment relation} $\vdash$ in ${\mathcal
  P}_\bullet(\Pi)^I$ is: $\label{eq:entailment0} \varphi,\psi \in
  {\mathcal P}_\bullet(\Pi)^I\,,\,\varphi \vdash \psi \Leftrightarrow
  \exists t\in {\mathrm{QP}}\;\forall i \in I \quad t\perp
  \varphi(i)\rightarrow_\bullet\psi(i),$ for $\varphi,\psi:I\to {\mathcal
    P}_\bullet(\Pi)$. An element $t \in \Phi$ as above is said to be
  ``a realizer of the entailment $\varphi\vdash\psi$''.
\item \label{lem:aks-iord-monot} For any function $f:J\to I$, $f^*$ is
  the monotonic map $f^*:({{\mathcal P}_\bullet}(\Pi)^I,\vdash)\to
  ({{\mathcal P}_\bullet}(\Pi)^J,\vdash),\, \varphi\mapsto \varphi\circ f$.
\item \label{def:aks-iord2}\label{lem:aks-iord-func} The preceding constructions yield an
indexed preorder:
\[ {\bf P}_\bullet({\mathcal K}):{\bf Set}^{\operatorname{op}}\rightarrow {\bf Ord},\qquad I\mapsto
({{\mathcal P}}(\Pi)^I,\vdash),\qquad f\mapsto f^*.
\]
\end{enumerate}

\begin{rema}
\begin{enumerate}  
\item Notice that the entailment relation above, could have been defined
  using the arrow $\rightarrow$, because $t\perp
  \varphi(i)\rightarrow \psi(i)$ if and only if $t\perp
  \varphi(i)\rightarrow_{\bullet}{}\psi(i)$, compare with Observation
  \ref{remark:equal}.
\item In particular from this follows that if we change in the above
  definition ${\mathcal P}_\bullet(\Pi)$ by ${\mathcal P}(\Pi)$, we may
  define the indexed preorder ${\bf P}(\mathcal K)$ and in this situation we
  have a natural inclusion ${\bf P}_{\bullet}(\mathcal K)
  \hookrightarrow {\bf P}(\mathcal K)$. 
\end{enumerate}
\end{rema}

\begin{lema} The canonical inclusion ${\bf P}_{\bullet}(\mathcal K)
\hookrightarrow {\bf P}(\mathcal K)$ is an equivalence of indexed
preorders.
\end{lema}
\begin{proof} By Lemma~\ref{lem:equiv-iord} it suffices to show
that the inclusion $({\mathcal
  P}_{\bullet}(\Pi)^I,\vdash)\hookrightarrow({{\mathcal
    P}}(\Pi)^I,\vdash)$ is an equivalence for all sets $I$ and this is
the content of Lemma \ref{lem:eqpo}.
\end{proof}
%\section{From $\mathcal {AKS}$s to
%$\koca$s} \label{subsection:akstooca}
\item In this Paragraph we prove the equivalence of all the relevant
  tripos constructed previously in this paper.
\begin{rema}\label{def:oca-to-islat} Let 
${\mathcal A}$ be a $\foca$: 
\begin{enumerate}
\item 
The \emph{entailment relation} is defined by $\label{eq:entailment}
\varphi\vdash \psi\Leftrightarrow \exists r\in \Phi\;\forall i\in
I\,\, r(\varphi(i))\leq\psi(i),$ for $\varphi,\psi:I\to A$. An element
$r \in \Phi$ as above is said to be ``a realizer of the entailment
$\varphi\vdash\psi$''. The entailment relation is a preorder on $A^I$,
and $(A^I,\vdash)$ is a meet-semi-lattice with: $\top:I\to A$;
$\top(i)=\top={\ck}$ and $(\varphi\wedge\psi)(i)= \varphi (i) \wedge
\psi (i)$.
\item \label{lem:oca-islat-monot} For any function $f:J\to I$, we
  define map $f^*:(A^I,\vdash)\to (A^J,\vdash),\,\, \varphi\mapsto
  \varphi\circ f$ that preserves meets and is monotonic.
\item \label{lem:oca-islat-func} Define the indexed meet-semi-lattice:
${\bf P}({\mathcal A}):{\bf Set}^{\operatorname{op}}\to{\bf SLat},\,\, I\mapsto
(A^I,\vdash),\,\, f\mapsto f^*$.
\item \label{theo:main} If ${\mathcal A}=
(A,\leq,\rapp,\rimp,\Phi,{\ck},{\cs})$ is a $\ioca$, then
  ${\bf P}(\mathcal A)$ is a tripos.  
\end{enumerate}
The last assertion can be proved in an identical manner than the
corresponding result for a $\ioca$, compare with \cite[Section 5,
  Theorem 5.8]{kn:ocar}.
\end{rema}
%%  In this paragraph, after recalling the construction ${\mathcal K}
%%  \mapsto {\mathcal A}_{{\mathcal K}\bullet}$ appearing in
%%  \ref{item:noadjunctor} that produces a $\ioca$ from an ${\mathcal
%%    {AKS}}$ we show that the indexed preorder associated to the
%%  ${\mathcal {AKS}}$ and the tripos associated to the $\ioca$ are
%%  isomorphic.

%% If we have an ${\mathcal {AKS}}, {\mathcal
%%   K}=(\Lambda,\Pi,\Perp,\mathrm{push},\mathrm{app},
%% \cK,\cS,\mathrm{QP})$ we constructed in item \ref{item:noadjunctor}
%% the $\ioca$,  ${\mathcal A}_{{\mathcal K}\bullet}=(A,\leq,\rapp,\rimp,{\ck},{\cs},\Phi)$ as:
%% \begin{enumerate}
%% \item $(A,\leq)=({\mathcal P}_{\bullet}(\Pi),\supseteq)$;
%% \item ${\rapp}(P,Q)=P \circ_{\bullet} Q ,\,\, {\mathrm{imp}}(P,Q)=P
%%   \rightarrow_{\bullet} Q$;
%% \item $\ck_{\bullet}=\{\cE\cK\}^\perp$, $\cs_{\bullet}=\{\cE((\cB\cE)\cS)\}^\perp$;
%% \item $\Phi=\{P \in {\mathcal P}_{\bullet}(\Pi):\, \exists t\in {\mathrm{QP}}. \,
%% t\perp P\}$.
%% \end{enumerate} 
%% If $a,b \in A$ we write $ab:={\rapp}(a,b)$ and $a \to b:={\rimp}(a,b)$.

%% We recall the following important theorem from \cite{kn:streicher} and
%% write down a short proof for later use. 

% The reader should be aware of
% a slight change of notation that is necessary for precision in the
% theorem below: we write 
% ${\mathcal P}(\mathcal A_\mathcal K)$ instead of ${\mathcal P}(A,\Phi)$ and
% ${\mathcal P}_\bot(\mathcal K)$ instead of ${\mathcal P}_\bot\mathcal K$.

\begin{theo}\label{theo:pakstooca2} Let ${\mathcal K}$  
  and ${\mathcal A}_{{\mathcal K}\bullet}$ be as in Paragraph \ref{item:noadjunctor}.  Then, the
  associated indexed preorders ${\bf P}_{\bullet}({\mathcal K})$
  and ${\bf P}(\mathcal A_{{\mathcal K}\bullet})$ are
  isomorphic.
\end{theo}
\begin{proof} The proof is a direct consequence of the fact that 
the Heyting preorders ${\mathcal H}_{{\mathcal K}\bullet}$ and
${\mathcal H}_{\mathcal A_{{\mathcal K}\bullet}}$ are isomorphic, see
Theorem \ref{theo:pakstooca}.
%% In both cases the predicates on a set $I$ are functions
%% $\varphi,\psi:I\to {\mathcal P}_{\bullet}(\Pi)$, so we only have to
%% check that the two definitions of entailment coincide. The entailment
%% in ${\bf P}({\mathcal A}_{{\mathcal K}\bullet})$ is given by: $\exists
%% P\in \Phi\;\forall i\in I \,.\, P\varphi(i)\leq\psi(i)$, which using
%% the full adjunction can be formulated equivalently as: $\exists P\in
%% \Phi\;\forall i\in I\,\, P\leq\varphi(i)\to_{\bullet}\psi(i)$.

%% As to the equivalence we have that:
%% \begin{eqnarray*} & \exists P\in \Phi\;\forall i\in I\,\,
%% P\leq\varphi(i)\to_{\bullet}\psi(i)\\ \Leftrightarrow\quad & \exists
%% t\in {\mathrm{QP}}\;\forall i\in I\,\, \{t\}^\perp\supseteq
%% ({}^\perp\varphi(i).\psi(i))^{\widehat{}}\\ \Leftrightarrow\,\, &
%% \exists t\in {\mathrm{QP}}\;\forall i\in I\,\, t\perp
%% ({}^\perp\varphi(i).\psi(i))^{\widehat{}} \\ \Leftrightarrow\quad &
%% \exists t\in {\mathrm{QP}}\;\forall i\in I\,.\,
%% t\perp {}^\perp\varphi(i).\psi(i)\\ \Leftrightarrow\quad & \exists
%% t\in {\mathrm{QP}}\;\forall i\in I\,\, t\perp{} \varphi(i)\rightarrow
%% \psi(i)\\
%% \end{eqnarray*} and the last line is the definition of entailment in
%% ${\bf P}_{\bullet}(\mathcal K)$.
\end{proof}

\begin{theo}\label{theo:equivtripos}
Let ${\mathcal A}$ and ${\mathcal K}_{\mathcal A\bullet}$ be as in
Paragraph \ref{item:inverse}. Then, the associated indexed triposes
${{\bf P}}({\mathcal A})$ and ${{\bf P}}(\mathcal A_{\mathcal K_{\mathcal A\bullet}\bullet})$ are equivalent.
%% Definitions
%%   ~\ref{lem:oca-islat}-\ref{lem:oca-islat-func} and
%%   in~\ref{lem:aks-iord}-\ref{lem:aks-iord-func} respectively).
\end{theo}
\begin{proof} The equivalence of the triposes follow immediately from the commutativity --up to equivalence-- of the diagram: 

\[\xymatrix{\mathcal{AKS} \ar[rd]_-{\mathcal S}&&\mathcal{{}^{\mathcal
      F}OCA}\ar[ld]^-{\mathcal H}\ar[ll]_{ {\mathcal
        K}_{\mathcal A \bullet}\leftarrow {\mathcal A}}\\ &{\bf
      HPO}&}\] as guaranteed by Theorem \ref{theo:main3}. 
\end{proof}
\end{list}

\end{document}